\documentclass[10pt]{amsart}

\usepackage[english]{babel}
\usepackage{amsmath}
\usepackage{amsfonts}
\usepackage{amssymb}
\usepackage{amsthm}
\usepackage{epsfig}
\usepackage{graphicx}
\usepackage{url}
\usepackage[colorlinks=true,linkcolor=blue,citecolor=blue]{hyperref}
\usepackage{changes}
\usepackage{tikz}
\usepackage[strict]{changepage}

\newtheorem{teor}{Theorem}[section]
\newtheorem{lema}[teor]{Lemma}
\newtheorem{corollary}[teor]{Corollary}
\newtheorem{prop}[teor]{Proposition}
\newtheorem{conjecture}[teor]{Conjecture}
\newtheorem{question}[teor]{Question}
\newtheorem*{remark*}{Remark}

\newcounter{theor}

\newtheorem{theorem}[theor]{Theorem}
\newtheorem{proposition}[theor]{Proposition}

\theoremstyle{definition}
\newtheorem{definition}[teor]{Definition}
\newtheorem{example}[teor]{Example}
\newtheorem{remark}[teor]{Remark}

\def\conv{\mathop\mathrm{conv}\nolimits}

\def\B{\mathbb{B}}
\def\s{\mathbb{S}}

\def\L{\mathcal{L}}

\def\R{\mathbb{R}}
\def\N{\mathbb{N}}
\def\Z{\mathbb{Z}}
\def\V{\mathrm{V}}
\def\vol{\mathrm{vol}}

\def\W{\mathrm{W}}

\def\esc#1{\left\langle #1\right\rangle}

\newcommand{\bola}{\mathbb{B}_n}
\newcommand{\enorm}[1]{\left|{#1}\right|}
\newcommand{\dlat}{\mathrm{d}}
\newcommand{\di}{d}

\def\deriv#1#2{\frac{\dlat}{\dlat t}^{\!\!\!#1}\Big|_{t=#2}\,}

\newcommand{\la}{\lambda}

\numberwithin{equation}{section}

\definechangesauthor[name=Jesus, color=blue]{J}
\definechangesauthor[name=Manuel, color=purple]{M}

\begin{document}

\title[Brunn-Minkowski inequalities]{Brunn-Minkowski inequalities in product Metric Measure Spaces}

\author{Manuel Ritor\'e}
\address{Departamento de Geometr\'ia y Topolog\'ia, Facultad de Ciencias, Universidad de Granada, 18071 Granada, Spain}
\email{ritore@ugr.es}

\author{Jes\'us Yepes Nicol\'as}
\address{Departamento de Matem\'aticas, Universidad de Le\'on, Campus de Vegazana, 24071 Le\'on, Spain}
\email{jyepn@unileon.es}

\thanks{First author is supported by MICINN-FEDER grant
MTM2013-48371-C2-1-P and Junta de Andaluc\'ia grants FQM-325 and
P09-FQM-5088.
Second author is partially supported by ICMAT Severo Ochoa project SEV-2015-0554 (MINECO) and by ``Programa
de Ayudas a Grupos de Excelencia de la Regi\'on de Murcia'', Fundaci\'on
S\'eneca, 19901/GERM/15.}

\date{May 4, 2017}



\begin{abstract}
Given one metric measure space $X$ satisfying a linear Brunn-Min\-kows\-ki inequality, and a second one $Y$ satisfying a Brunn-Minkowski inequality with exponent $p\ge -1$, we prove that the product $X\times Y$ with the standard product distance and measure satisfies a Brunn-Minkowski inequality of order $1/(1+p^{-1})$ under mild conditions on the measures and the assumption that the distances are strictly intrinsic. The same result holds when we consider restricted classes of sets. We also prove that a linear Brunn-Minkowski inequality is obtained in $X\times Y$ when $Y$ satisfies a Pr\'ekopa-Leindler inequality.

In particular, we show that the classical Brunn-Minkowski inequality holds for any pair of weakly unconditional sets in $\R^n$ (i.e., those containing the projection of every point in the set onto every coordinate subspace) when we consider the standard distance and the product measure of $n$ one-dimensional real measures with positively decreasing densities. This yields an improvement of the class of sets satisfying the Gaussian Brunn-Minkowski inequality.

Furthermore, associated isoperimetric inequalities as well as recently obtained Brunn-Minkowski's inequalities are derived from our results.
\end{abstract}

\maketitle

\section{Introduction}\label{s:intro}
\thispagestyle{empty}

The $n$-dimensional
volume of a set $M$ in the $n$-dimensional Euclidean space $\R^n$ (i.e., its $n$-dimensional Lebesgue
measure) is denoted by $\vol(M)$, or $\vol_n(M)$ if the distinction of the
dimension is useful. The symbol
$\bola$ stands for the $n$-dimensional closed unit ball with respect to the Euclidean norm $\enorm{\,\cdot\,}$.

Relevant families of subsets of Euclidean space used in this work are those of unconditional and weakly unconditional sets:
a subset $A\subset\R^n$ is said to be \emph{unconditional} if
for every $(x_1,\dots,x_n)\in A$ and every $(\epsilon_1,\dots,\epsilon_n)\in[-1,1]^n$ one has
\[(\epsilon_1x_1,\dots,\epsilon_nx_n)\in A.\]
In a similar way, we will say that $A$ is \emph{weakly unconditional} (see Figure \ref{figure: weakly_unc}) if for every $(x_1,\dots,x_n)\in A$ and every $(\epsilon_1,\dots,\epsilon_n)\in\{0,1\}^n$ one has
\[(\epsilon_1x_1,\dots,\epsilon_nx_n)\in A.\]

Weakly unconditional sets are those for which the projection of every point in the set onto any coordinate subspace is again contained in the set.
Equivalently, a set $A$ is weakly unconditional if and only if e\-ve\-ry non-empty $1$-section of $A$, through parallel lines to the coordinate axes, contains the origin (identifying the corresponding $1$-dimensional affine subspace with its direction; cf. Figure \ref{figure: weakly_unc}).

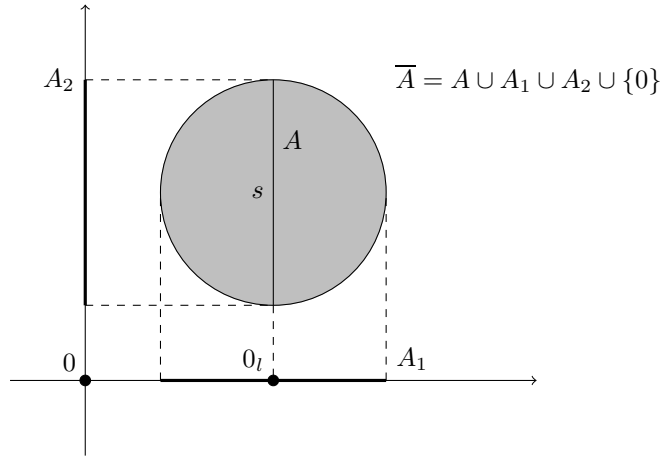
\begin{figure}[h]
\begin{tikzpicture}
\filldraw [lightgray] (2.5,2.5) circle [radius=1.5] ;
\draw (2.5,2.5) circle [radius=1.5] ;
\draw[dashed] (0,4) -- (2.5,4) ;
\draw[dashed] (0,1) -- (2.5,1) ;
\draw[dashed] (4,0) -- (4,2.5) ;
\draw[dashed] (1,0) -- (1,2.5) ;
\draw (2.5,4) -- (2.5,1) ;
\draw[dashed] (2.5,1) -- (2.5,0) ;
\draw (2.5,2.5) node[left]{$s$} ;
\draw (2.5,3.2) node[right]{$A$} ;
\draw[->] (-1,0) -- (6,0) ;
\draw[->] (0,-1) -- (0,5) ;
\filldraw [black] (0,0) circle [radius=2pt] node[above left]{$0$} ;
\filldraw [black] (2.5,0) circle [radius=2pt] node[above left]{$0_l$} ;
\draw[very thick] (1,0) -- (4,0) node[above right]{$A_1$} ;
\draw[very thick] (0,1) -- (0,4) node[left]{$A_2$} ;
\draw (4,4) node[right]{$\overline{A}=A\cup A_1\cup A_2\cup\{0\}$} ;
\end{tikzpicture}
\caption{The weakly unconditional hull $\overline{A}$ of $A$. Every $1$-section $s$ of $\overline{A}$, through a parallel line $l$ to a coordinate axis, contains the origin $0_l$.}\label{figure: weakly_unc}
\end{figure}

Given an arbitrary non-empty set $B\subset\R^n$, $\overline{B}$ will denote its \emph{weakly unconditional hull} (i.e., the intersection of all weakly unconditional sets containing $B$), which is just the union of $B$ with every projection of it onto any coordinate subspace. The unconditional hull of $B$ is defined in a similar way, see Figure~\ref{figure: weakly_unc_hull}.

\begin{figure}[h]
\begin{tikzpicture}
\filldraw [lightgray] (-1.5,-1.5) -- (1.5,-1.5) -- (1.5,1.5) -- (-1.5,1.5) -- (-1.5,-1.5) ;
\draw (-1.5,-1.5) -- (1.5,-1.5) -- (1.5,1.5) -- (-1.5,1.5) -- (-1.5,-1.5) ;
\filldraw [black] (1.5,1.5) circle [radius=2pt] ;
\draw (1.5,1.5) circle [radius=4pt]  node[right=2pt]{$p$};
\filldraw [black] (1.5,0) circle [radius=2pt] ;
\filldraw [black] (0,1.5) circle [radius=2pt] ;
\filldraw [black] (0,0) circle [radius=2pt]  node[above right]{$0$};
\draw[->] (-2,0) -- (2.5,0) node[above]{$x$} ;
\draw[->] (0,-2) -- (0,2.5) node[right]{$y$};
\end{tikzpicture}
\caption{The unconditional hull of the point $p\in\R^2$ is the gray square, while the weakly unconditional hull of $p$ consists of the four marked points}
\label{figure: weakly_unc_hull}
\end{figure}
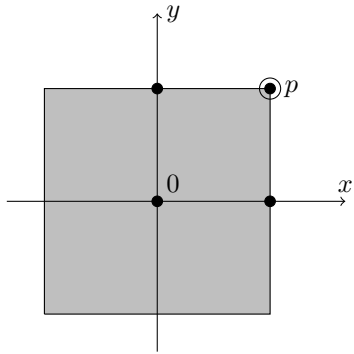

It is worth mentioning that an unconditional set in $\R$ is an interval symmetric with respect to the origin, and that a weakly unconditional set in $\R$ is just a set containing the origin.

Another notion used in this paper is that of \emph{positively decreasing} function.
We say that a non-negative function $f:\R\longrightarrow\R_{\geq0}$ is positively decreasing if the functions
$t\mapsto f(t)$, $t\mapsto f(-t)$ are decreasing (i.e., non-increasing) on $[0,\infty)$.

\smallskip

The Minkowski sum of two non-empty sets $A, B\subset\R^n$ denotes the classical vector addition of them:
$A+B=\{a+b:\, a\in A, \, b\in B\}$.
It is natural to wonder about the possibility of relating the volume of the Minkowski sum of two sets in terms of their volumes; this is the statement of the \emph{Brunn-Minkowski inequality}.
Indeed, taking $\lambda\in(0,1)$ and $A$ and $B$ two non-empty (Lebesgue) measurable subsets of $\R^n$ such that their linear combination
\begin{equation}\label{e:lincomb}
(1-\lambda)A+\lambda B=\{(1-\lambda)a+\lambda b:\,a\in A,\, b\in B\}
\end{equation}
is also measurable, then the Brunn-Minkowski inequality ensures that
\begin{equation}\label{e:BM}
\vol\bigl((1-\lambda)A+\lambda B\bigr)^{1/n}\geq
(1-\lambda)\vol(A)^{1/n}+\lambda\vol(B)^{1/n}.
\end{equation}

The above inequality is sometimes introduced in a weaker form (although they are actually equivalent because of the homogeneity of the volume),
often referred to as its multiplicative or dimension free version:
\begin{equation}\label{e:mult_BM}
\vol\bigl((1-\lambda)A+\lambda
B\bigr)\geq\vol(A)^{1-\lambda}\vol(B)^{\lambda}.
\end{equation}

The Brunn-Minkowski inequality admits an equivalent analytic version, \emph{the Pr\'e\-ko\-pa-Leindler inequality},
originally proved in \cite{Prekopa} and \cite{Leindler}, from where inequality \eqref{e:mult_BM} can be immediately obtained.
\begin{theorem}[Pr\'ekopa-Leindler Inequality]\label{t:PrekopaLeindler}
Let $\lambda\in (0,1)$ and let $f,g,h:\R^n\longrightarrow\R_{\geq0}$ be non-negative measurable
functions such that, for any $x,y\in\R^n$,
\begin{equation*}\label{e:PrekopaLeindlerCondicion}
h\bigl((1-\lambda) x+\lambda y\bigr)\geq f(x)^{1-\lambda}g(y)^{\lambda}.
\end{equation*}
Then
\begin{equation*}\label{e:PrekopaLeindler}
\int_{\R^n}h \,\dlat x\geq
\left(\int_{\R^n}f\,\dlat x\right)^{1-\lambda}\left(\int_{\R^n}g\,\dlat x\right)^{\lambda}.
\end{equation*}
\end{theorem}

The Pr\'ekopa-Leindler inequality is naturally connected to log-concave functions (i.e., functions of the form $e^{-u}$ where $u\,:\,\R^n\longrightarrow\R\cup\{\infty\}$ is convex). In other words, a non-negative function $\phi$ is log-concave if $\phi((1-\lambda)x+\lambda y)\geq \phi(x)^{1-\lambda}\phi(y)^\lambda$ for any $x,y\in\R^n$ and $\lambda\in(0,1)$. In the same way, $f$ is \emph{quasi-concave} if $\phi((1-\lambda)x+\lambda y)\geq \min\{\phi(x),\phi(y)\}$, or equivalently if its superlevel sets are convex. In particular, if $\mu$ is a measure on $\R^n$ such that $\dlat\mu(x)=\phi(x)\,\dlat x$, where $\phi$ is log-concave (a measure like this will be called a \emph{log-concave measure}) then, by Theorem \ref{t:PrekopaLeindler},
\begin{equation}\label{e:BM_mult_meas}
\mu\bigl((1-\lambda)A+\lambda
B\bigr)\geq\mu(A)^{1-\lambda}\mu(B)^{\lambda}.
\end{equation}
for any $A,B\subset\R^n$ measurable such that $(1-\lambda)A+\lambda B$ is so, for $\lambda\in(0,1)$.

However, although multiplicative Brunn-Minkowski type inequalities \eqref{e:BM_mult_meas} may be easily obtained (it is enough to consider log-concave measures), something very different occurs when dealing with the $(1/n)$-form of the Brunn-Minkowski inequality \eqref{e:BM}.

A relevant space where the general Brunn-Minkowski inequality does not hold for arbitrary sets is the so-called \emph{Gauss space}: $\R^n$ endowed with the standard $n$-dimensional Gaussian measure given by
$$\dlat\gamma_n(x)=\frac{1}{(2\pi)^{n/2}}e^{\frac{-\enorm{x}^2}{2}}\dlat x.$$
Indeed, since $1/(2\pi)^{n/2}\,e^{\frac{-\enorm{x}^2}{2}}$
is log-concave, $\gamma_n$ satisfies \eqref{e:BM_mult_meas}; however, it is easy to see that
\begin{equation}\label{e:GaussBM}
\gamma_n\bigl((1-\lambda)A+\lambda B\bigr)^{1/n}\geq
(1-\lambda)\gamma_n(A)^{1/n}+\lambda\gamma_n(B)^{1/n}
\end{equation}
does not hold \emph{in general}, as it was pointed out in \cite[Section~7]{GaZv} via $A=\bola$ and $B=x+\bola$ for $x$ large enough.

To this respect, since a condition on the `position of the sets' is clearly needed, the authors conjectured (Question $7.1$) the following.
\begin{conjecture}[\cite{GaZv}]\label{conjecture}
The Gaussian Brunn-Minkowski inequality \eqref{e:GaussBM} holds for any closed convex sets $A, B$ such that $0\in A\cap B$.
\end{conjecture}
In the same paper \cite{GaZv} the authors were able to verify this conjecture when both sets are coordinate boxes containing the origin and when one of them is a slab containing the origin: $[-a_1,a_2]\times\R^{n-1}$ for $a_1,a_2>0$ (as well as when the sets are both dilates of the same \emph{centrally symmetric} closed convex set).

In \cite{NaTk}, the authors showed that Conjecture \ref{conjecture} is in general not true: it is enough to consider
\begin{equation}\label{e:ContraejGBM}
   A=\{(x,y)\in\R^2:\, y\geq\enorm{x}\tan\alpha\}, \quad B=A+(0,-\varepsilon),
\end{equation}
for $\varepsilon>0$ small enough and $\alpha<\pi/2$ sufficiently close to $\pi/2$.

At this point, it is natural to ask for the weakest assumption in relation to the position of the sets needed to obtain Brunn-Minkowski's inequality \eqref{e:GaussBM}.
In a recent paper, \cite[Theorem~1]{LiMaNaZv}, the authors have proved in a very elegant way that \eqref{e:GaussBM} holds for the case of \emph{unconditional sets}. More precisely, the authors show:
\begin{theorem}[\cite{LiMaNaZv}]\label{t:LiMaNaZv}
Let $\mu=\mu_1\times\dots\times\mu_n$ be a product measure on $\R^n$ such that $\mu_i$ is the measure given by $\dlat\mu_i(x)=\phi_i(x)\,\dlat x$,
where $\phi_i:\R\longrightarrow\R_{\geq0}$ is a positively decreasing even function, $i=1,\dots,n$.

Let $\lambda\in(0,1)$ and let $\emptyset\neq A,B\subset\R^n$ be unconditional measurable sets so that $(1-\lambda)A+\lambda B$ is also measurable.
Then
\begin{equation*}
\mu((1-\lambda)A+\lambda B)^{1/n}\geq(1-\lambda)\mu(A)^{1/n}+\lambda\mu(B)^{1/n}.
\end{equation*}
\end{theorem}

In the same paper the authors pose the following:
\begin{question}
Can one remove the assumption
of unconditionality in the Gaussian Brunn-Minkowski inequality?
\end{question}

Here we give a positive answer to this question by showing that it is enough to consider weakly unconditional sets.
Furthermore, this may be regarded as a positive answer to (the necessary slight reformulation of) Conjecture \ref{conjecture}, see Remark \ref{r:GaussBM}.
One of our main results reads, in the more general setting of product of measures with positively decreasing densities, as follows:

\begin{teor}\label{t:BM_prod_quasi-conc}
Let $\mu=\mu_1\times\dots\times\mu_n$ be a product measure on $\R^n$ such that $\mu_i$ is the
measure given by $\dlat\mu_i(x)=\phi_i(x)\,\dlat x$,
where $\phi_i:\R\longrightarrow\R_{\geq0}$ is a positively decreasing function, $i=1,\dots,n$.

Let $\lambda\in(0,1)$ and let $\emptyset\neq A,B\subset\R^n$ be weakly unconditional measurable sets such that $(1-\lambda)A+\lambda B$ is also measurable. Then
\begin{equation}\label{e:BM_prod_quasi-conc}
\mu((1-\lambda)A+\lambda B)^{1/n}\geq(1-\lambda)\mu(A)^{1/n}+\lambda\mu(B)^{1/n}.
\end{equation}
\end{teor}

We would like to point out that, in contrast to Theorem \ref{t:LiMaNaZv}, not only the ``symmetry assumption'' on the sets may be removed but also that of the density functions involved in the product measure.

\smallskip

Our approach is based on obtaining new Brunn-Minkowski inequalities from certain simpler spaces that previously satisfy some Brunn-Minkowski type inequality. To this end, we have the following definition.

\begin{definition}\label{d:1}
Let $(X,\di,\mu)$ be a metric measure space. We will say that it \emph{satisfies the Brunn-Minkowski inequality} with respect to the parameter $p\in\R\cup\{\pm\infty\}$, and we denote it by BM($p$) for short, if
\begin{equation}\label{e:BM(p)}
\mu(C)\geq\bigl((1-\lambda)\mu(A)^p+\lambda\mu(B)^p\bigr)^{1/p}
\end{equation}
holds for all $\lambda\in(0,1)$ and any non-empty measurable sets $A, B, C$ with $\mu(A)\mu(B)>0$ such that $C\supset(1-\lambda)A\star_\di\lambda B$ (for a precise formulation of the operation $\star_\di$ see Definition \ref{d:lincomb_metric} in Subsection~\ref{ss:distance}).

If \eqref{e:BM(p)} holds for general non-empty measurable sets $A, B, C$ such that $C\supset(1-\lambda)A\star_\di\lambda B$, without the restriction $\mu(A)\mu(B)>0$, we say that
$(X,\di,\mu)$ \emph{satisfies the general Brunn-Minkowski inequality} with respect to the parameter $p\in\R\cup\{\pm\infty\}$. We denote it by $\overline{\mathrm{BM}}$($p$) for short.

In the same way, we say that a certain family $\mathcal{F}\subset \mathcal{P}(X)$ of measurable sets satisfies $\mathrm{BM}(p)$ (resp.~$\overline{\mathrm{BM}}(p)$) if the above definition holds when dealing with sets $A, B, C\in\mathcal{F}$.

On the other hand, we will say that $(X,\di,\mu)$ satisfies the Pr\'ekopa-Leindler inequality, denoted by PL for short, if
for any $\lambda\in (0,1)$ and non-negative $\mu$-measurable
functions $f,g,h:X\longrightarrow\R_{\geq0}$ such that
\begin{equation*}
h(z)\geq f(x)^{1-\lambda}g(y)^{\lambda},
\end{equation*}
for all $x,y\in X$, and $z\in(1-\lambda)\{x\}\star_\di\lambda\{y\}$,
then
\begin{equation*}
\int_{X}h\,\dlat\mu\geq
\left(\int_{X}f\,\dlat\mu\right)^{1-\lambda}\left(\int_{X}g\,\dlat\mu\right)^{\lambda}.
\end{equation*}

In the same way, a function $\phi:X\longrightarrow\R_{\geq0}$ will be said to be log-concave if
\begin{equation*}
\phi(z)\geq \phi(x)^{1-\lambda}\phi(y)^{\lambda},
\end{equation*}
for all $x,y\in X$, $z\in(1-\lambda)\{x\}\star_\di\lambda\{y\}$, and any $\lambda\in(0,1)$.
\end{definition}

With this notation we obtain our second main result:

\begin{teor}\label{t:BM_ProdMetSpa}
Let $(X,\di_X,\mu_X)$, $(Y,\di_Y,\mu_Y)$ be metric measure spaces where $\mu_X$ is $\sigma$-finite, $\mu_Y$ is locally finite and $\mu_{X\times Y}$ is a Radon measure.

If $(X,\di_X,\mu_X)$, $(Y,\di_Y,\mu_Y)$ satisfy $\overline{\mathrm{BM}}(1)$ and $\mathrm{BM}(p)$, respectively, for some $p\geq-1$,
then $(X\times Y,\di_{X\times Y},\mu_{X\times Y})$ satisfies $\mathrm{BM}\bigl(1/(1+p^{-1})\bigr)$.
\end{teor}

In the statement of Theorem~\ref{t:BM_ProdMetSpa}, we follow the usual convention of assuming that the quantity $1/(1+p^{-1})$ is equal to $0$ when $p=0$, and equal to $1$ when $p=\infty$. As pointed out to the authors by Luca Rizzi, Theorem~\ref{t:BM_ProdMetSpa} can be generalized when $Y$ satisfies the weighted BM$(1/n,N/n)$ Brunn-Minkowski inequality
\[
\mu(C)^{1/n}\ge (1-\la)^{N/n}\mu(A)^{1/n}+\la^{N/n}\mu(B)^{1/n}.
\]
In this case the product $X\times Y$ satisfies the weighted Brunn-Minkowski inequality BM$(1/(n+1),(N+1)/(n+1))$. See Theorem~\ref{t:BM_ProdMetSpa_gen}. This result yields, in the setting of product metric spaces, the same weighted Brunn-Minkowski inequality to the one in corank $1$ Carnot groups obtained in \cite{balogh2}.

The linear Brunn-Minkowski inequality $\overline{\mathrm{BM}}(1)$ plays a relevant role and becomes one of the main ingredients in this paper. Thus we also focus on this type of inequalities. To this aim, we have the following result in the setting of product metric spaces endowed with log-concave measures.

\begin{teor}\label{t:BM_lin1_MetSpa}
Let $(X,\di_X,\mu_X)$, $(Y,\di_Y,\mu_Y)$ be metric measure spaces which satisfy $\overline{\mathrm{BM}}(1)$ and $\mathrm{PL}$, respectively. Let $\mu$ be the measure on $X\times Y$ given by
$\dlat\mu(x,y)=\phi(x,y)\dlat\mu_{X\times Y}(x,y)$ where $\phi$ is a log-concave function. Assume that $\mu_X, \mu_Y$ are $\sigma$-finite.
Let $\lambda\in(0,1)$ and let $A,B\subset X\times Y$ be non-empty measurable sets such that
$(1-\lambda)A\star\lambda B$ is also measurable and so that
\begin{equation*}
\sup_{x\in X}\int_Y\chi_{_A}(x,y)\phi(x,y)\,\dlat\mu_Y=\sup_{x\in X}\int_Y\chi_{_B}(x,y)\phi(x,y)\,\dlat\mu_Y.
\end{equation*}
Then
\begin{equation*}
\mu\bigl((1-\lambda)A\star\lambda B\bigr)\geq(1-\lambda)\mu(A)+\lambda\,\mu(B).
\end{equation*}
\end{teor}

The paper is organized as follows. Section \ref{s:background} is devoted to collecting some
definitions and preliminary considerations. In Section \ref{s:main} we prove Theorems \ref{t:BM_ProdMetSpa} and \ref{t:BM_lin1_MetSpa} (in fact, a more general version of the latter) as well as some related result in Theorem \ref{t:linBM_ProdMetSpa}. Next, in Section \ref{s:BM_Rn}, we particularize the results obtained in the precedent section to get some Brunn-Minkowski type inequalities for product measures on $\R^n$. Among others, we show Theorem \ref{t:BM_prod_quasi-conc} as well as we study its associated isoperimetric type inequality, Theorem \ref{t:isop}. Finally, in Section \ref{s:applications}, we present some other applications that can be derived from the results of Section \ref{s:main}.

\section{Background material and auxiliary results}\label{s:background}

Another equivalent version of the Brunn-Minkowski inequality that appears in the literature is the following:
\begin{equation}\label{e:min_BM}
\vol\bigl((1-\lambda)A+\lambda
B\bigr)\geq\min\{\vol(A),\,\vol(B)\}.
\end{equation}

The equivalent expressions of the Brunn-Minkowski inequality, \eqref{e:BM}, \eqref{e:mult_BM}, \eqref{e:min_BM}, can be rewritten in terms of the $p$-th mean of two non-negative numbers,
where $p$ is a parameter varying in $\R\cup\{\pm\infty\}$.
We recall this definition, for which we follow \cite{BL} (regarding a general reference for $p$-th means of non-negative numbers, we refer also to the classic text of Hardy, Littlewood and P\'olya \cite{HaLiPo}, as well as to the excellent handbook about means by Bullen \cite{Bu}).

Consider first the case $p\in\R$ and $p\ne0$; given $a,b\ge0$ such that $ab\ne0$ and $\lambda\in(0,1)$, we define
\[
M_p(a,b,\lambda)=((1-\lambda)a^p+\lambda b^p)^{1/p}.
\]
For $p=0$ we set
$$
M_0(a,b,\lambda)=a^{1-\lambda}b^\lambda
$$
and, to complete the picture, for $p=\pm \infty$ we define $M_\infty(a,b,\lambda)=\max\{a,b\}$ and
$M_{-\infty}(a,b,\lambda)=\min\{a,b\}$. Finally, if $ab=0$, we will define $M_p(a,b,\lambda)=0$ for all $p\in\R\cup\{\pm\infty\}$. Note that $M_p(a,b,\lambda)=0$, if $ab=0$, is redundant for all $p\leq0$, however it is relevant for $p>0$ (as we will briefly comment later on). Moreover, one can easily check that, for any $p\in\R\cup\{\pm\infty\}$, $M_p(a,b,\lambda)=\lim_{q\to p}M_q(a,b,\lambda)$. Furthermore, for $p\neq0$, we will allow that $a$, $b$ take the value $\infty$ and in that case, as usual, $M_p(a,b,\lambda)$ will be the value that is obtained ``by continuity''.

With this definition, inequalities \eqref{e:BM}, \eqref{e:mult_BM}, \eqref{e:min_BM} can be rephrased as
\begin{equation}\label{e:BM_mean}
\vol\bigl((1-\lambda)A+\lambda
B\bigr)\geq M_p\bigl(\vol(A),\vol(B),\lambda\bigr),
\end{equation}
where $p$ takes the values $1/n$, $0$ and $-\infty$, respectively. We notice that, in the case $p=1/n$, and with the above considered notation, we are only taking into account sets $A,B$ such that $\vol(A)\vol(B)>0$. This distinction is usual in the literature when dealing with Brunn-Minkowski type inequalities and we shall discuss about it later (cf. Proposition \ref{p:BMzeromeas_sets} and the precedent paragraph).
For this reason, we will maintain both notations, $M_p(a,b,\lambda)$ and $((1-\lambda)a^p+\lambda b^p)^{1/p}$, along the paper (when $p=0,\pm\infty$, the latter must be understood again as the values that are obtained ``by continuity'').

\smallskip

Here, we are interested in studying Brunn-Minkowski type inequalities in the context of a \emph{metric measure space}
(i.e., a set $X$ endowed with a distance $\di$ and a measure $\mu$). To this end, two ``elements'' playing a relevant role in the above family of inequalities \eqref{e:BM_mean} must be studied: the \emph{operation} (the Minkowski sum, $+$, in \eqref{e:BM_mean}) and the \emph{measure} (the Lebesgue measure, $\vol(\cdot)$, in \eqref{e:BM_mean}) together with the appropriate $p$-th mean according to the ``geometry'' of the space.

\subsection{The operation: the distance comes into play}\label{ss:distance}
Since we intend to study Brunn-Minkowski type inequalities in spaces without a linear structure, first we should interpret \eqref{e:lincomb}
in terms of the distance associated to the Euclidean norm $\enorm{\,\cdot\,}$.
To this aim, we notice that, given $x,y,z\in\R^n$ and $\lambda\in(0,1)$, the relation $z=(1-\lambda)x+\lambda y$ holds
if and only if
\begin{equation}\label{e:x,y,z}
\enorm{z-x}=\lambda\enorm{x-y}, \, \enorm{z-y}=(1-\lambda)\enorm{x-y}.
\end{equation}
In other words,
\begin{equation*}
\{(1-\lambda)x+\lambda y\}=\overline{B}\bigl(x,\lambda \enorm{x-y}\bigr)\cap\overline{B}\bigl(y,(1-\lambda)\enorm{x-y}\bigr)
\end{equation*}
(here $\overline{B}(a,r)$ denotes the ($n$-dimensional) closed Euclidean ball of radius $r$ centered at $a$, i.e., $\overline{B}(a,r)=a+r\bola$).
The above intersection is non-empty in any normed space $X$ (considering, as natural, balls with respect to the distance induced by the given norm) since $(1-\lambda)x+\lambda y$ (for $x,y\in X$ and $\lambda\in(0,1)$) always satisfies \eqref{e:x,y,z}. However, it does not necessarily contain a unique point: consider, for instance, $\R^n$ endowed with the infinity norm $\enorm{\,\cdot\,}_{\infty}$, or the Heisenberg group endowed with its standard Carnot-Carath\'eodory distance. This does not suppose a big inconvenient for us; the problem is when dealing with metric spaces for which the latter intersection may be empty. A trivial example of this would be $\Z^n$ endowed with the Euclidean metric of $\R^n$ restricted on $\Z^n$. Since this space is not \emph{path-connected} this obstacle cannot be overcome, i.e., there does not exist an equivalent distance satisfying the above non-empty intersection condition (for any pair of points and any $\lambda\in[0,1$]). An example of a path-connected space for which a given distance has also the above-mentioned peculiarity is the unit circle $\s^{1}$ in the plane with the restriction of (the distance given by) $\enorm{\,\cdot\,}$ to $\s^{1}$. Here the solution is to consider the \emph{intrinsic metric} on $\s^{1}$ associated to its \emph{length structure}.
These considerations lead us to the following definition (we recommend the reader to \cite[Chapter~2]{BuBuIv} for enlightening discussions on these topics).
\begin{definition}
Let $(X,\di)$ be a metric space. We will say that $\di$ is a strictly intrinsic distance if for any $x,y\in X$ the closed balls $\overline{B}_\di(x,r_1)$, $\overline{B}_\di(y,r_2)$ have a non-empty intersection provided that $r_1+r_2=\di(x,y)$.
\end{definition}
An interesting class of metric spaces with strictly intrinsic distances is the one of complete length spaces endowed with their associated distances. Following the proof of \cite[Theorem~2.4.16]{BuBuIv}, one may show that a metric space $(X,\di)$, where $\di$ is a strictly intrinsic distance, needs to be path-connected.

For the sake of brevity in our exposition, from now on a \emph{metric space} will mean a set $X$ provided with a strictly intrinsic distance $\di$.
Under this consideration, we may extend \eqref{e:lincomb} to the context of a metric space (cf. \eqref{e:x,y,z}).
\begin{definition}\label{d:lincomb_metric}
Let $(X,\di)$ be a metric space. If $\lambda\in(0,1)$ and $A,B$ are two non-empty subsets of $X$,
the ``$\di$-convex combination'' $(1-\lambda)A\star_\di\lambda B$ of $A$ and $B$ will be the non-empty set given by
\begin{equation}\label{e:lincomb_metric}
\begin{split}
(1-\lambda)A\star_\di\lambda B=\bigl\{z\in X:{}
 &\di(z,a)=\lambda\di(a,b),\\ & \di(z,b)=(1-\lambda)\di(a,b),
a\in A, b\in B\bigr\}.
\end{split}
\end{equation}
\end{definition}
This notion of $\di$-convex combination has already appeared in the literature.
It can be found in \cite[Definition~1.1]{Ju} and, in the setting of Riemannian manifolds, denominated as the `barycenter of two points', in \cite[p.~29]{Cor}.

When the distance is clear, the index of $\star_\di$ will be omitted and we shall write just $(1-\lambda)A\star\lambda B$ to denote $(1-\lambda)A\star_\di\lambda B$.

When the metric $\di$ is such that the intersection
\[\overline{B}_\di\bigl(x,\lambda\di(x,y)\bigr)\cap\overline{B}_\di\bigl(y,(1-\lambda)\di(x,y)\bigr)=\{z\}\]
(for any $x,y\in X$ and $\lambda\in(0,1)$),
it allows us to define an operation on $X$,
$+_\di:X\times X\times[0,1]\longrightarrow X$, given by $+_\di(x,y,\lambda)=z$. For simplicity we write $(1-\lambda)x+_\di\lambda y=+_\di(x,y,\lambda)$ and further, for $A, B\subset X$,
\[(1-\lambda)A+_\di\lambda B=\{(1-\lambda)a+_\di\lambda b:\,a\in A,\, b\in B\}=(1-\lambda)A\star_\di\lambda B.\]

An easy way to construct such a metric space endowed with the above-mentioned operation $+_\di$ is as follows.
\begin{definition}\label{d:+_d_assoc_to_phi}
Let $\phi:X\longrightarrow\R^n$ be an injective function on a set $X$. Then the function $\di:X\times X\longrightarrow\R_{\geq0}$ given by $\di(x,y)=\enorm{\phi(x)-\phi(y)}$ is a distance on $X$.
Moreover, assuming that $\di$ is strictly intrinsic (which implies here that, for any $x,y\in X$ and $\lambda\in(0,1)$, there exists $z\in X$ such that $\phi(z)=(1-\lambda)\phi(x)+\lambda\phi(y)$) the operation $+_\di$ is well-defined and we have
\begin{equation*}
(1-\lambda)x+_\di\lambda y=\phi^{-1}\bigl((1-\lambda)\phi(x)+\lambda\phi(y)\bigr).
\end{equation*}
\end{definition}

When working with product spaces, the $\di$-convex combination of two subsets (cf.~\eqref{e:lincomb_metric}) might not be clear since there is not a ``sole'' distance $\di$ associated to the product metric space. So, it is convenient to clarify which distance will be ``fixed'' for the product space.

To this aim, let $(X,\di_X)$ and $(Y,\di_Y)$ be metric spaces and consider the product space $X\times Y$.
One can define a distance $\di_{X\times Y}$ on $X\times Y$ (whose induced topology agrees with the product topology) as follows:
$$\di_{X\times Y}\bigl((x_1,x_2), (y_1,y_2)\bigr)=||(\di_X(x_1,y_1),\di_Y(x_2,y_2))||$$
where $||\cdot||$ is a norm in $\R^2$.
When taking the Euclidean norm $\enorm{\,\cdot\,}$, this distance is the so-called \emph{product metric}, which will be denoted from now on as $\rho$.

For our purposes (cf. \eqref{e:BM_mean}), amongst all the distances $\di$ on $X\times Y$, we are interested in considering the smallest possible $\di$-convex combination. First we notice that, from the definition of $\di_{X\times Y}$, and for any non-empty $A,B\subset X\times Y$ and $\lambda\in(0,1)$, we get
\begin{equation}\label{e: inclus_convcomb_distances}
(1-\lambda)A\star_{\di_{X\times Y}}\!\lambda B\supset(1-\lambda)A\star^\prime\lambda B,
\end{equation}
where $(1-\lambda)A\star^\prime\lambda B$ is the set given by
\begin{equation*}\label{e:operationXY_1}
\begin{split}
(1-\lambda)A\star^\prime\lambda B=\bigl\{z=&(z_1,z_2)\in X\times Y :{}
\di_X(z_1,a_1)=\lambda\di_X(a_1,b_1), \\
&\di_X(z_1,b_1)=(1-\lambda)\di_X(a_1,b_1),\\
&\di_Y(z_2,a_2)=\lambda\di_Y(a_2,b_2), \\
&\di_Y(z_2,b_2)=(1-\lambda)\di_Y(a_2,b_2),\\
&\quad\text{ for some } (a_1,a_2)\in A, (b_1,b_2)\in B\bigr\}.
\end{split}
\end{equation*}
The following result shows that, in a sense, the most ``convenient'' distance on $X\times Y$ (cf.
\eqref{e: inclus_convcomb_distances}) is the product metric $\rho$.

\begin{prop}\label{p:starProdSpac}
Let $(X,\di_X)$ and $(Y,\di_Y)$ be metric spaces and let $\rho$ be the product metric on $X\times Y$.
If $A,B$ are non-empty subsets of $X\times Y$ and $\lambda\in(0,1)$ then
$(1-\lambda)A\star_\rho\lambda B=(1-\lambda)A\star^\prime\lambda B$.
\end{prop}
\begin{proof}
For the sake of simplicity, we will write $\di_1=\di_X$, $\di_2=\di_Y$. If we denote by
$A_i=\di_i(a_i,b_i)$, $B_i=\di_i(z_i,a_i)$, $C_i=\di_i(z_i,b_i)$, for $i=1,2$, and we put
$A=(A_1,A_2)$, $B=(B_1,B_2)$, $C=(C_1,C_2) \in\R_{\geq0}^2$, we have that the defining conditions
of $\star_\rho$ are
\begin{equation}\label{e:cond1_2expres_oper}
\enorm{B}_2=\lambda\enorm{A}_2, \quad \enorm{C}_2=(1-\lambda)\enorm{A}_2.
\end{equation}
From the triangle inequality (for $\di_i$) we get
\begin{equation}\label{e:cond2_2expres_oper}
A_i\leq B_i+C_i, \, \text{ for } i=1,2.
\end{equation}
By \eqref{e:cond2_2expres_oper}, the triangle inequality (for $\enorm{\,\cdot\,}_2$) and \eqref{e:cond1_2expres_oper}, respectively, we have
\begin{equation*}
\enorm{A}_2\leq \enorm{B+C}_2\leq\enorm{B}_2+\enorm{C}_2=\enorm{A}_2 .
\end{equation*}
Thus, from the equality case of the triangle inequality for $\enorm{\,\cdot\,}_2$, we obtain $C=rB$ where, by \eqref{e:cond1_2expres_oper}, $r=\frac{1-\lambda}{\lambda}$. Finally, from the equality case of \eqref{e:cond2_2expres_oper}, we get
\begin{equation*}
A_i = B_i+C_i=(1+r)B_i=\frac{1}{\lambda}B_i, \, \text{ for } i=1,2.
\end{equation*}
This, together with the already obtained $C=rB$, implies that $B_i=\lambda A_i$ and $C_i=(1-\lambda)A_i$, $i=1,2$.
It finishes the proof.
\end{proof}
For simplicity, from now on (unless we say explicitly the opposite), when working with product metric spaces, $\star$ will stand for $\star_\rho$ (i.e., $\star=\star_\rho=\star^\prime$).

\subsection{The role of the measure}
As we have commented before, we intend to study Brunn-Minkowski type inequalities. Even considering $\R^n$ endowed with the classical Minkowski addition $+$, it would be naive to try to generalize \eqref{e:BM_mean} when replacing the volume with an arbitrary measure $\mu$. Indeed, if we consider $\R$ provided with the measure $\mu$ given by $\dlat\mu(x)=x^2\,\dlat x$ and we take $A=[1,2]$, $B=-A$, then we clearly have $\mu(A)=\mu(B)>\mu([-1/2,1/2])=\mu\bigl((A+B)/2\bigr)$, which prevents any possible Brunn-Minkowski type inequality (cf.~\eqref{e:BM_mean}).
Other examples of spaces where certain Brunn-Minkowsi type inequalities do not hold in general (although the multiplicative Brunn-Minkowski inequality is immediately fulfilled (by Theorem \ref{t:PrekopaLeindler})) are related to the so-called \emph{log-concave measures}.

Thus, since there arise some pathologies even when considering $\R^n$ with the (Minkowski addition and the) most important probability measure in $\R^n$, in order to derive a certain Brunn-Minkowski inequality with respect to a fixed mean $M_p$ (cf.~\eqref{e:BM_mean}), we should either impose some conditions on the measure or consider certain subfamilies of sets (or sometimes, as commented above, both of them).

\smallskip

For the sake of completeness, we collect some definitions and results that to be used throughout our discussion.
\begin{definition}
A measure space is a pair $(X,\mu)$ (that is, we omit for simplicity the $\sigma$-algebra $\Sigma$ which contains, by assumption, the $\sigma$-algebra of all Borel sets in $X$), where $\mu$ is a measure on the metric space $(X,\di)$ that is assumed to be \emph{complete}. Such a measure is said to
be \emph{Borel}, whereas the triple $(X,\di,\mu)$ is called a metric measure space.

When considering the product measure space associated to the measure spaces
$(X,\mu_{X})$ and $(Y,\mu_{Y})$, we often denote by $\mu_{X\times Y}$, instead of $\mu_X\times\mu_Y$, the completion of the usual product measure of $\mu_X$ and $\mu_Y$. Moreover, in this case, $(X,\mu_{X})$ and $(Y,\mu_{Y})$ are assumed to be measure spaces so that $\mu_{X\times Y}$ is also Borel.
\end{definition}

Regarding the measurability of the sets, it is worth noting that the assumption that $A$ and $B$ are measurable
is not sufficient to guarantee that $(1-\lambda) A\star\lambda B$ is measurable, as happens with the Lebesgue measure
in relation to the Minkowski addition (see \cite[Section 10]{G}).

The following results are classical and can be found in any book on measure theory (e.g. \cite{Coh}).
\begin{proposition}\label{p:meas_union_inters}
Let $(X,\mu)$ be a measure space.
If $(A_n)_{n\in\N}$ is a decreasing sequence of $\mu$-measurable sets such that $\mu(A_n)<+\infty$ for some $n\in\N$, then \[\mu\left(\bigcap_{n=1}^{+\infty} A_n\right)=\lim_n \mu(A_n).\]
\end{proposition}

\begin{theorem}[Fubini-Tonelli's theorem]\label{t:Fubini}
Let $(X,\mu_X)$, $(Y,\mu_Y)$ be $\sigma$-finite measure spaces and let $f:X\times Y\longrightarrow\R_{\geq0}$ be a non-negative measurable function. Then
\begin{equation*}
\begin{split}
\int_{X\times Y} f(x,y)\,\dlat \mu_{X\times Y}(x,y)
&=\int_X\left(\int_Y f(x,y)\,\dlat \mu_Y(y)\right)\,\dlat \mu_X(x)\\
&=\int_Y\left(\int_X f(x,y)\,\dlat \mu_X(x)\right)\,\dlat \mu_Y(y).
\end{split}
\end{equation*}
Therefore, for any non-negative measurable function $h:X\longrightarrow\R_{\geq0}$, Cavalieri's Principle
\begin{equation*}
\int_X \,h(x)\,\dlat \mu_X(x)
= \int_0^{+\infty}\mu_X\bigl(\{x\in X: h(x)\geq t\}\bigr)\, \dlat t
\end{equation*}
holds.
\end{theorem}
In the following, when dealing with product spaces, we use the following notation:
given $C\subset X\times Y$, for any $t\in X$ we denote by
\[C(t)=\{y\in Y: (t,y)\in C\}.\]
\begin{definition}
Let $(X,\di,\mu)$ be a metric measure space.
We recall that $\mu$ is said to be \emph{locally finite} if for every point $x\in X$ there exists $r_x>0$ such that $\mu\bigl(B_\di(x,r_x)\bigr)<\infty$.
We recall that $\mu$ is \emph{inner regular} if the measure of any set can be approximated from within by compact subsets of $X$: if for any measurable set $A$
\[\mu(A)=\sup\{\mu(K):\,K\subset A, K \text{ compact}\}\]
holds.
If $\mu$ is both inner regular and locally finite, it is called a \emph{Radon} measure.
Finally, $\mu$ is said to be \emph{strictly positive} if $\mu\bigl(B_\di(x,r)\bigr)>0$ for all $r>0$ and any $x\in X$.
In the same way, we say that $\mu$ is strictly positive around $x\in X$, if $\mu\bigl(B_\di(x,r)\bigr)>0$ for all $r>0$.
\end{definition}

Regarding Definition \ref{d:1}, we note that
$(X,\di,\mu)$ satisfies the Brunn-Minkowski inequality BM($p$) if
\begin{equation*}
\mu(C)\geq M_p\bigl(\mu(A),\mu(B),\lambda\bigr)
\end{equation*}
holds for all $\lambda\in(0,1)$ and any measurable sets $A, B, C$ such that $C\supset(1-\lambda)A\star\lambda B$.

We notice that for all $p\leq0$, a space satisfies BM($p$) if and only if it also satisfies $\overline{\mathrm{BM}}$($p$). The following result shows that when dealing with some ``special'' spaces, both notions are also equivalent for $p>0$.
Before showing it, we would like to point out that this property is trivially fulfilled (for $p=1/n$) in $\R^n$ endowed with the Lebesgue measure $\vol(\cdot)$ (cf.~\eqref{e:BM}), as an easy consequence of both the translation invariance and the homogeneity (of degree $n$) of the volume. However, when considering a different measure $\mu$, in principle, the case $\mu(A)\mu(B)=0$ cannot be easily obtained.

\begin{prop}\label{p:BMzeromeas_sets}
Let $(X,\di,\mu)$ be a metric measure space, where $(X,\di)$ is locally compact and $\mu$ is a strictly positive Radon measure. If $(X,\di,\mu)$ satisfies $\mathrm{BM}(p)$ (with respect to the parameter $p>0$) then it satisfies $\overline{\mathrm{BM}}(p)$.
\end{prop}
\begin{proof}
Since $\mu$ is Radon then it is inner regular and thus it is enough to show the statement of this result for arbitrary non-empty compact sets $A$ and $B$.
Moreover, we will consider the case in which one of the sets, say $B$, has measure zero whereas the other one, $A$, has positive measure (it is immediate otherwise).

Let $b_0\in B$ and consider, for any $n\in\N$, $B_n=B\cup \overline{B}_\di\bigl(b_0,1/n\bigr)$.
On the one hand, $(B_n)_{n\in\N}$ is clearly a decreasing sequence and
\[\bigcap_{n=1}^{+\infty} B_n=B\cup\bigcap_{n=1}^{+\infty}\overline{B}_\di\left(b_0,\frac{1}{n}\right)=B\cup\{b_0\}=B.\]
Moreover, since $X$ is locally compact, we may assume, without loss of generality, that $\overline{B}_\di\bigl(b_0,\frac{1}{n}\bigr)$ is compact (and then also $B_n$) for all $n\in\N$. Hence
$(1-\lambda)A\star\lambda B_n$ is compact and thus $\mu\bigl((1-\lambda)A\star\lambda B_n\bigr)<+\infty$ because
$\mu$ is locally finite. So, from Proposition \ref{p:meas_union_inters}, we have
\[\lim_n\mu\bigl((1-\lambda)A\star\lambda B_n\bigr)=\mu\left(\bigcap_{n=1}^{+\infty}\bigl((1-\lambda)A\star\lambda B_n\bigr)\right).\]

On the other hand, we may assume that $\mu\left(\overline{B}_\di\bigl(b_0,1/n\bigr)\right)<+\infty$ for all $n\in\N$ because $\mu$ is locally finite, and thus the same holds for $\mu(B_n)$.
Furthermore, since $\mu$ is strictly positive, $\mu\left(\overline{B}_\di\bigl(b_0,\frac{1}{n}\bigr)\right)>0$ and hence $\mu(B_n)>0$. Therefore, by hypothesis,
\begin{equation*}
\mu\bigl((1-\lambda)A\star\lambda B_n\bigr)\geq\bigl((1-\lambda)\mu(A)^p+\lambda\mu(B_n)^p\bigr)^{1/p},
\end{equation*}
and taking limits on both sides (and using Proposition \ref{p:meas_union_inters}) we get
\begin{equation*}
\lim_n\mu\bigl((1-\lambda)A\star\lambda B_n\bigr)\geq\bigl((1-\lambda)\mu(A)^p+\lambda\mu(B)^p\bigr)^{1/p}=(1-\lambda)^{1/p}\mu(A).
\end{equation*}

So, we must just check that
$\bigcap_{n\in\N}\bigl((1-\lambda)A\star\lambda B_n\bigr)\subset(1-\lambda)A\star\lambda B$ (furthermore, we have equality since the reverse inclusion is trivially fulfilled).
To this aim, let $z\in\bigcap_{n\in\N}\bigl((1-\lambda)A\star\lambda B_n\bigr)$. Then, for each $n\in\N$, there exist
$a_n\in A$, $b_n\in B_n$ such that
\begin{equation}\label{e: cond_z_an_bn}
\di(z,a_n)=\lambda\di(a_n,b_n), \qquad \di(z,b_n)=(1-\lambda)\di(a_n,b_n).
\end{equation}

Since $(B_n)_n$ is decreasing and $X$ is locally compact, there exists $n_0\in\N$ such that $B_n$ is compact for $n\ge n_0$. As $A$ is compact we may assume, taking subsequences if necessary, that $(a_n)_n$, $(b_n)_n$ are convergent, with limits $a\in A$ and $b\in B_{n_0}$, respectively. Repeating the argument with the tails of the sequence $(b_n)_n$, we get that $b\in B_n$ for all $n\ge n_0$ and hence $b\in B$. Now, taking limits on \eqref{e: cond_z_an_bn}, we may assert that $z\in(1-\lambda)A\star\lambda B$. This concludes the proof.
\end{proof}

Now we collect some examples of spaces or families of sets satisfying linear Brunn-Minkowski inequalities, since these spaces/families will be very useful along this paper (cf.~Theorems \ref{t:linBM_ProdMetSpa} and \ref{t:BM_ProdMetSpa}).
\begin{example}\label{ex:1}
Let $X=\alpha(I)$ where $I=[0,L]\subset\R$ (for some $L>0$) and $\alpha:I\longrightarrow\R^n$ is an injective
function. Let $\di$ be the distance defined, as in Definition \ref{d:+_d_assoc_to_phi}, via $\alpha^{-1}$ by the equality $\di(x,y)=\enorm{\alpha^{-1}(x)-\alpha^{-1}(y)}$. Let $\mu$ be the pushforward of the one-dimensional Lebesgue measure $\vol_1$ on $\R$.
Then, by the Brunn-Minkowski inequality in $\R$ (cf.~\eqref{e:BM}), we have
\begin{equation*}
\begin{split}
\mu((1-\lambda)A\star\lambda B)
&=\mu\bigl(\alpha\bigl((1-\lambda)\alpha^{-1}(A)+\lambda\alpha^{-1}(B)\bigr)\bigr)\\
&=\vol_1\bigl((1-\lambda)\alpha^{-1}(A)+\lambda\alpha^{-1}(B)\bigr)\\
&\geq (1-\lambda)\vol_1\bigl(\alpha^{-1}(A)\bigr)+\lambda\vol_1\bigl(\alpha^{-1}(B)\bigr)\\
&=(1-\lambda)\mu(A)+\lambda\mu(B).
\end{split}
\end{equation*}
\end{example}

The following linear version of the Brunn-Minkowski inequality, due to Bonnesen, can be found in the literature (e.g.~{\cite{BF,Gi,Oh}}).
From now on, $\L^n_{n-1}$ will be the set of vectorial hyperplanes of $\R^n$, the orthogonal projection of a set $A$ onto $H$ will be denoted by $A|H$, and by a convex body we will mean a (non-empty) compact convex set.
\begin{example}\label{ex:1.2}
Let $K,L\subset\R^n$ be convex bodies for which there exists
$H\in\L^n_{n-1}$ such that either $\vol_{n-1}(K|H)=\vol_{n-1}(L|H)$ or
$$\quad \max_{x\in H^\perp}\vol_{n-1}\bigl(K\cap(x+H)\bigr)=\max_{x\in
H^\perp}\vol_{n-1}\bigl(L\cap(x+H)\bigr).$$
Then, for all
$\lambda\in(0,1)$,
\begin{equation*}\label{e:BMproye_vol}
\vol\bigl((1-\lambda)K+\lambda L\bigr)
\geq(1-\lambda)\vol(K)+\lambda\vol(L).
\end{equation*}
\end{example}

We notice that a positively decreasing function $\phi:\R\longrightarrow\R_{\geq0}$ is quasi-concave and furthermore $\phi(0)=\sup_{x\in\R}\phi(x)$. Thus, as a consequence of \cite[Theorem~4.1]{CoSGYN} we get the following example, which will play a relevant role along this paper.
\begin{example}\label{ex:2}
Let $\mu$ be the measure on $\R$ given by $\dlat\mu(x)=\phi(x)\dlat x$, where $\phi:\R\longrightarrow\R_{\geq0}$ is positively decreasing. Let $\lambda\in(0,1)$ and let $A, B\subset\R$ be measurable sets \emph{containing the origin} and such that $(1-\lambda)A+\lambda B$ is also measurable.
Then $\overline{\mathrm{BM}}(1)$ holds:
\begin{equation*}
\mu((1-\lambda)A+\lambda B)\geq(1-\lambda)\mu(A)+\lambda\mu(B).
\end{equation*}
We recall that weakly unconditional sets in $\R$ are those containing the origin. Hence $\overline{\mathrm{BM}}(1)$ holds for weakly unconditional sets. On the other hand, if $\mu$ is a strictly positive finite measure, taking $A=[-1,1]$, and $B_i=[i,i+1]$ for all $i\in\N$, we have $\lim_i\mu((1-\lambda)A+\lambda B_i)=\lim_i\mu(B_i)=0$, and we conclude that $\overline{\mathrm{BM}}(1)$ cannot hold for $A$ and $B_i$ when $i$ is large. This shows the necessity  of taking weakly unconditional sets to ensure the validity of $\overline{\mathrm{BM}}(1)$.
\end{example}
The following result shows that, when dealing with measures $\mu$ on the real line given by $\dlat\mu(x)=\phi(x)\dlat x$, where $\phi$ is a \emph{continuous} function, the converse is also true. That is, the fact that the linear Brunn-Minkowski inequality holds for any pair of measurable sets containing the origin characterizes the ``nature'' of $\phi$.
\begin{prop}\label{p:linBMimplies_quasi}
Let $\mu$ be the measure on $\R$ given by $\dlat\mu(x)=\phi(x)\dlat x$, where $\phi$ is a (non-negative) continuous function.
If
\begin{equation}\label{e:linBMimplies_quasi}
\mu((1-\lambda)A+\lambda B)\geq(1-\lambda)\mu(A)+\lambda\mu(B)
\end{equation}
holds for any measurable sets $A, B\subset\R$ containing the origin and all $\lambda\in(0,1)$, and so that $(1-\lambda)A+\lambda B$ is also measurable, then $\phi$ is positively decreasing.
\end{prop}
\begin{proof}
Let $F:\R\longrightarrow\R$ be the function given by
\[F(x)=\int_0^x\phi(t)\,\dlat t.\]
Fix $x,y>0$ and take $A=[0,x]$ and $B=[0,y]$. Then, from \eqref{e:linBMimplies_quasi}, we get
$F\bigl((1-\lambda)x+\lambda y\bigr)\geq(1-\lambda)F(x)+\lambda F(y)$. Since it is true for arbitrary $x,y\in\R_{>0}$ and $\lambda\in(0,1)$, we may assure that $F$ is concave on $\R_{>0}$. In the same way, we obtain that $F$ is convex on $\R_{<0}$. Moreover, since $\phi$ is continuous, by the fundamental theorem of calculus we get $F^\prime(x)=\phi(x)$ for all $x\in\R$. Now, the concavity of $F$ on $\R_{>0}$ (resp. the convexity of $F$ on $\R_{<0}$) implies that $\phi(x)=F^\prime(x)$ is decreasing on $\R_{>0}$ (resp. $\phi(x)=F^\prime(x)$ is increasing on $\R_{<0}$).
The result is now concluded from the continuity of $\phi$.
\end{proof}

\begin{remark}
The condition that the sets $A$ and $B$ contain the origin in Proposition~\ref{p:linBMimplies_quasi} is necessary but weak enough to allow a wide range of measures to satisfy $\overline{\mathrm{BM}}(1)$. Indeed, if we assume a more restrictive condition, such as $\overline{\mathrm{BM}}(1)$ holds for any pair of Euclidean balls in $\R$, then the measure $\mu$ is a constant multiple of the Lebesgue measure $\vol_1$ (see \cite[Theorem~1.2]{YN} and \cite{Borell}). This shows once again the necessity of taking weakly unconditional sets to ensure $\overline{\mathrm{BM}}(1)$ when dealing with arbitrary measures.
\end{remark}

\section{Brunn-Minkowski inequalities in product measure spaces}\label{s:main}

One of the best known proofs of the Brunn-Minkowski inequality \eqref{e:BM} is a classical version due to Kneser and S\"uss \cite{KnSu}, which is also reproduced in \cite{BF} and \cite{Sch2}. Hadwiger and Ohmann \cite{HaOh} gave a particularly beautiful proof for the compact setting, which is also reproduced in the survey by Gardner \cite[Section~4]{G} and in Burago and Zalgaller's monograph \cite{BZ} (see also the references therein). Knothe \cite{Kn} gave a proof of the Brunn-Minkowski inequality using a volume-preserving map. Both Kneser-S\"uss' and Hadwiger-Ohmann's proofs have an inductive flavor; the first one in the dimension $n$ whereas the second one in the number of boxes that determine the set (the general case is then obtained by ``approximation''). However, both the \emph{homogeneity} and the \emph{translation invariance} of the volume play a crucial role in these proofs. Thus, it is not possible to imitate these proofs in the context of metric measure spaces.

Nevertheless, one may obtain Brunn-Minkowski inequalities in further spaces where the measure does not satisfy neither homogeneity nor translation invariance. To this aim, here we will exploit some classical analytic tools that are often used when dealing with some stuff relative to the Pr\'ekopa-Leindler inequality, Theorem \ref{t:PrekopaLeindler} (e.g.~\cite{Borell,BL,DaUr,HMacb}. 
First we will show a linear Brunn-Minkowski inequality on product metric spaces in the spirit of Bonnesen's inequality in Example \ref{ex:1.2}.

\begin{teor}\label{t:linBM_ProdMetSpa}
Let $(X,\di_X,\mu_X)$, $(Y,\di_Y,\mu_Y)$ be metric measure spaces which satisfy $\overline{\mathrm{BM}}(1)$ and $\mathrm{BM}(p)$, respectively, where $p\in\R\cup\{\pm\infty\}$ and $\mu_X, \mu_Y$ are $\sigma$-finite. Let $\lambda\in(0,1)$ and let $A,B\subset X\times Y$ be non-empty measurable sets such that
$(1-\lambda)A\star\lambda B$ is also measurable and so that
\begin{equation}\label{e:igual_seccion_max}
\sup_{x\in X}\mu_Y(A(x))=\sup_{x\in X}\mu_Y(B(x)).
\end{equation}
Then
\begin{equation*}
\mu_{X\times Y}\bigl((1-\lambda)A\star\lambda B\bigr)\geq(1-\lambda)\mu_{X\times Y}(A)+\lambda\,\mu_{X\times Y}(B).
\end{equation*}
\end{teor}

\begin{proof}
Along this proof we will use the same symbol $\star$ to denote the operation on any of the three spaces $X$, $Y$ and $X\times Y$. We shall denote the product measure $\mu_{X\times Y}$ simply by $\mu$.
We set $\alpha=\sup_{x\in X}\mu_Y(A(x))=\sup_{x\in X}\mu_Y(B(x))\in\R_{\geq0}\cup\{+\infty\}$.

In case $\alpha=0$, Fubini-Tonelli's Theorem and \eqref{e:igual_seccion_max} imply that $\mu(A)=\mu(B)=0$ and the linear Brunn-Minkowski inequality holds trivially. Hence we may assume that $\alpha>0$.

For any pair of points $s,r\in X$ such that $\mu_Y(A(s))\mu_Y(B(r))>0$, the sets $A(r)$ and $B(s)$ are non-empty. Take a point $t_\lambda\in (1-\la)\{s\}\star\la\{r\}$. This implies that $\di_X(t_\lambda,s)=\lambda\di_X(s,r)$, $\di_X(t_\lambda,r)=(1-\lambda)\di_X(s,r)$. From these equalities the inclusion
\[
\bigl((1-\lambda)A\star\lambda B\bigr)(t_\lambda)\supset(1-\lambda)A(s)\star\lambda B(r)
\]
follows immediately and, since BM($p$) is satisfied on $Y$, we get
\begin{equation}\label{e:concavidad_secciones}
\begin{split}
\mu_Y\Bigl(\bigl((1-\lambda)A\star\lambda B\bigr)(t_\lambda)\Bigr)
&\geq\mu_Y\bigl((1-\lambda)A(s)\star\lambda B(r)\bigr)\\
&\geq M_p(\mu_Y(A(s)),\mu_Y(B(r)),\la).
\end{split}
\end{equation}

Consider now the non-negative functions $f,g,h:X\longrightarrow\R_{\geq0}$ defined by $f(t)=\mu_Y(A(t))$, $g(t)=\mu_Y(B(t))$, and $h(t)=\mu_Y\bigl(\bigl((1-\lambda)A\star\lambda B\bigr)(t)\bigr)$. We clearly have (cf.~\eqref{e:igual_seccion_max})
\begin{equation}\label{e:novacios}
\bigl\{x\in X:\,f(x)\geq t\bigr\},\,\bigl\{x\in X:\,g(x)\geq t\bigr\}\neq\emptyset,
\end{equation}
for all $0\leq t<\alpha$.

From \eqref{e:novacios} and \eqref{e:concavidad_secciones} we get
\begin{equation*}
\{x\in X: h(x)\geq t\}\supset(1-\lambda)\{x\in X:f(x)\geq t\}\star\lambda\{x\in X: g(x)\geq t\}
\end{equation*}
for all $0<t<\alpha$.
Therefore, by the linear Brunn-Minkowski inequality in $X$, we have
\begin{multline*}
\mu_X\bigl(\{x\in X: h(x)\geq t\}\bigr)
\geq(1-\lambda)\mu_X\bigl(\{x\in X: f(x)\geq t\}\bigr)\\ +\lambda\mu_X\bigl(\{x\in X: g(x)\geq t\}\bigr)
\end{multline*}
for all $0<t<\alpha$.

Finally, by the above inequality, and using Fubini-Tonelli's Theorem and Cavalieri's Principle, Theorem \ref{t:Fubini}, we get
\begin{equation*}
\begin{split}
\mu\bigl((1-\lambda)A\star\lambda B\bigr)&=\int_X \,h(x)\,\dlat \mu_X(x)
\\
&= \int_0^{+\infty}\mu_X\bigl(\{x\in X: h(x)\geq t\}\bigr)\, \dlat t
\\
&\geq \int_0^\alpha\mu_X\bigl(\{x\in X: h(x)\geq t\}\bigr)\, \dlat t
\\
&\geq \int_0^\alpha\Bigl((1-\lambda)\mu_X\bigl(\{x\in X: f(x)\geq t\}\bigr)
\\
&\hspace{0.2\textwidth}+\int_0^\alpha\lambda\mu_X\bigl(\{x\in X: g(x)\geq t\}\bigr)\Bigr)\, \dlat t
\\
&= (1-\lambda)\int_X f(x)\,\dlat \mu_X(x) + \lambda\int_X g(x)\,\dlat \mu_X(x)
\\
&=(1-\lambda)\mu(A)+\lambda\,\mu(B),
\end{split}
\end{equation*}
as desired.
\end{proof}

As an immediate consequence of the proof of the above result we get the following corollary.

\begin{corollary}\label{c:linBM_ProdMetSpa}
Let $(X,\di_X,\mu_X)$, $(Y,\di_Y,\mu_Y)$ be metric measure spaces for which there exist certain families $\mathcal{F}_X\subset \mathcal{P}(X)$, $\mathcal{F}_Y\subset \mathcal{P}(Y)$ that satisfy  $\overline{\mathrm{BM}}(1)$ and $\mathrm{BM}(p)$, respectively, where $p\in\R\cup\{\pm\infty\}$ and $\mu_X, \mu_Y$ are $\sigma$-finite.
Let $A,B\subset X\times Y$ be measurable sets such that $(1-\lambda)A\star\lambda B$ is also measurable and so that
\begin{equation*}
\sup_{x\in X}\mu_Y(A(x))=\sup_{x\in X}\mu_Y(B(x)).
\end{equation*}
If moreover $A,B$ satisfy:
\begin{enumerate}
  \item $A(t), B(t)\in\mathcal{F}_Y$ for all $t\in X$,

  \smallskip

  \item $\{x\in X:\mu_Y(A(x))\geq t\},\{x\in X:\mu_Y(B(x))\geq t\}\in\mathcal{F}_X$ for all $0<t<\sup_{x\in X}\mu_Y(A(x))$ (or a.e.),
\end{enumerate}
then
\begin{equation*}
\mu_{X\times Y}\bigl((1-\lambda)A\star\lambda B\bigr)\geq(1-\lambda)\mu_{X\times Y}(A)+\lambda\,\mu_{X\times Y}(B).
\end{equation*}
\end{corollary}

The above proof can be exploited to obtain a general Brunn-Minkowski inequality in the setting of product metric measure spaces. This is the content of Theorem \ref{t:BM_ProdMetSpa}.

\begin{proof}[Proof of Theorem \ref{t:BM_ProdMetSpa}]
Along this proof we will use the same symbol $\star$ to denote the operation on any of the three spaces $X$, $Y$ and $X\times Y$. We shall denote the product measure $\mu_{X\times Y}$ simply by $\mu$.

Let $A, B, C\subset X\times Y$ be measurable sets such that $\mu(A)\mu(B)>0$ and $C\supset(1-\lambda)A\star\lambda B$.
Since $\mu$ is a Radon measure we may assume, without loss of generality, that $A$ and $B$ are compact: to prove this, we choose two sequences of compact sets $(K_i)_i$ and $(L_i)_i$ with positive volume such that $K_i\subset A$, $L_i\subset B$ for all $i\in\N$, and
\[
\mu(A)=\lim_i\mu(K_i), \quad \mu(B)=\lim_i\mu(L_i).\]
Assuming that BM($1/(1+p^{-1})$) holds for the pair $(K_i,L_i)$ we clearly have
\begin{equation*}
\begin{split}
\mu(C)&\geq\mu((1-\lambda)K_i\star\lambda L_i)\\
&\geq M_{1/(1+p^{-1})}(\mu(K_i),\mu(L_i),\la).
\end{split}
\end{equation*}
Taking limits on both sides, we get BM($1/(1+p^{-1})$) for $\{A, B, C\}$.

We take the non-negative functions $f,g,h:X\longrightarrow\R_{\geq0}$ given by
\[f(t)=\frac{\mu_Y(A(t))}{\enorm{\mu_Y(A(\cdot))}_\infty}, \ g(t)=\frac{\mu_Y(B(t))}{\enorm{\mu_Y(B(\cdot))}_\infty},
\ h(t)=\frac{\mu_Y(C(t))}{C_p},\]
where
\[
C_p=M_p\big(\enorm{\mu_Y(A(\cdot))}_\infty, \enorm{\mu_Y(B(\cdot))}_\infty,\lambda\big).
\]
We notice that the above functions are well-defined: denominators are positive since $\mu(A)\mu(B)>0$, and they are finite because $\{y\in Y:\, (x,y)\in A \text{ for some } x\in X\}$ and $\{y\in Y:\, (x,y)\in B \text{ for some } x\in X\}$ are compact subsets of $Y$ and $\mu_Y$ is locally finite. Furthermore
\[\sup_{t\in X}f(t)=\sup_{t\in X}g(t)=1.\]

We show that, for any pair of points $s,r\in X$, and any $t_\la\in (1-\la)\{s\}\star\la\{r\}$, we have
\begin{equation}\label{e:condit_hfg}
h(t_\lambda)\geq\min\{f(s),g(r)\}.
\end{equation}
To check the validity of \eqref{e:condit_hfg}, it is enough to consider the case $\mu_Y(A(s))\mu_Y(B(r))>0$.
Hence, $A(s), B(r)$ are non-empty, the inclusions
\[
C(t_\la)\supset\bigl((1-\lambda)A\star\lambda B\bigr)(t_\lambda)\supset(1-\lambda)A(s)\star\lambda B(r)
\]
trivially hold, and we have
\begin{equation*}
C_ph(t_\la)=\mu_Y(C(t_\la))\ge M_p(\mu_Y(A(s)),\mu_Y(B(r)),\la).
\end{equation*}

Now, for $p\neq0,+\infty$, we obtain
\begin{equation*}
\begin{split}
M_p\big(\mu_Y(A(s)),\mu_Y(B(r)),\la\big)=&\bigl((1-\lambda)\mu_Y(A(s))^p+\lambda\mu_Y(B(r))^p\bigr)^{1/p}
\\
&=C_p\bigl((1-\theta)f(s)^p+\theta g(r)^p\bigr)^{1/p}\\
&\geq C_p\min\{f(s),g(r)\},
\end{split}
\end{equation*}
where $\theta=\displaystyle\frac{\lambda\enorm{\mu_Y(B(\cdot))}_\infty^p}{C_p^p}\in(0,1)$.

For the case $p=0$, we get
\begin{equation*}
\begin{split}
M_0\big(\mu_Y(A(s)),\mu_Y(B(r)),\la\big)&=
\mu_Y(A(s))^{1-\lambda}\mu_Y(B(r))^{\lambda}
\\
&=C_0 \, f(s)^{1-\lambda}g(r)^{\lambda}\\
&\geq C_0\min\{f(s),g(r)\}.
\end{split}
\end{equation*}

For $p=\infty$, we clearly have
\begin{align*}
M_{\infty}\big(\mu_Y(A(s)),\mu_Y(B(r)),\la\big)&=\max\{\mu_Y(A(s)),\mu_Y(B(r))\}
\\
&\geq C_{\infty}\,\min\{f(s),g(r)\}.
\end{align*}
Therefore, we have shown \eqref{e:condit_hfg}.

The definition of $f$ and $g$ implies that the sets
\[
\bigl\{x\in X:\,f(x)\geq t\bigr\},\,\bigl\{x\in X:\,g(x)\geq t\bigr\}
\]
are non-empty for any $t\in [0,1)$. Now, \eqref{e:condit_hfg} trivially implies
\[
\{x\in X: h(x)\geq t\}\supset(1-\lambda)\{x\in X:f(x)\geq t\}\star\lambda\{x\in X: g(x)\geq t\}
\]
and, since $X$ satisfies $\overline{\mathrm{BM}}(1)$, we have
\begin{multline*}
\mu_X\bigl(\{x\in X: h(x)\geq t\}\bigr)
\geq(1-\lambda)\mu_X\bigl(\{x\in X: f(x)\geq t\}\bigr)\\ +\lambda\mu_X\bigl(\{x\in X: g(x)\geq t\}\bigr)
\end{multline*}
for any $t\in [0,1)$.

Now, following similar steps to those described at the end of the proof of Theorem~\ref{t:linBM_ProdMetSpa}, we get
\begin{equation*}
\begin{split}
\mu\bigl((1-\lambda)A\star\lambda B\bigr)&=C_p\int_X h(x)\,\dlat \mu_X(x)\\
&\geq C_p\left((1-\lambda)\int_X f(x)\,\dlat \mu_X(x) + \lambda\int_X g(x)\,\dlat \mu_X(x)\right)\\
&=C_p\left((1-\lambda)\frac{\mu(A)}{\enorm{\mu_Y(A(\cdot))}_\infty}+
\lambda\frac{\mu(B)}{\enorm{\mu_Y(B(\cdot))}_\infty}\right).
\end{split}
\end{equation*}
The latter quantity is no smaller than $M_{1/(1+p^{-1})}(\mu(A),\mu(B),\la)$. When $p\neq 0,\infty$, this follows from the reverse H\"older inequality, \cite[Theorem~1, p.~178]{Bu},
\[
(1-\la)a_1b_1+\la a_2 b_2\ge \big((1-\la)a_1^{-p}+\la a_2^{-p}\big)^{-1/p}
\big((1-\la)b_1^q+\la b_2^q\big)^{1/q}
\]
where $q=1/(1+p^{-1})$ is the H\"older conjugate of $(-p)\le 1$, just by taking $a_1=|\mu_Y(A(\cdot))|_\infty^{-1}$, $a_2=|\mu_Y(B(\cdot))|_\infty^{-1}$, $b_1=\mu(A)$, $b_2=\mu(B)$.

The case $p=0$ follows from the Arithmetic-Geometric mean inequality and the case $p=\infty$ is immediate.
\end{proof}

Regarding Definition \ref{d:1}, we could say that a metric measure space $(X,\di,\mu)$  satisfies the Brunn-Minkowski inequality with respect to the parameters $p\in\R\cup\{\pm\infty\}$ and $q\in(1,\infty)$, BM($p, q$) for short, if
\begin{equation*}
\mu(C)\geq\bigl((1-\lambda)^q\mu(A)^p+\lambda^q\mu(B)^p\bigr)^{1/p}
\end{equation*}
holds for all $\lambda\in(0,1)$ and any non-empty measurable sets $A, B, C$ with $\mu(A)\mu(B)>0$ such that $C\supset(1-\lambda)A\star_\di\lambda B$.
Analogously, if the above condition holds for general non-empty measurable sets $A, B, C$ such that $C\supset(1-\lambda)A\star_\di\lambda B$, without the restriction $\mu(A)\mu(B)>0$, we say that $(X,\di,\mu)$ satisfies $\overline{\mathrm{BM}}(p,q)$.
This notion is sometimes useful (e.g. \cite{Ju}, \cite[Section~4]{balogh}, and the references therein) when the classical Brunn-Minkowski inequality BM($p$) does not hold.

In particular, given $n,N>0$, a metric measure space satisfies BM$(1/n,N/n)$ if
\[
\mu(C)^{1/n}\ge (1-\la)^{N/n}\mu(A)^{1/n}+\la^{N/n}\mu(B)^{1/n}
\]
holds for all measurable sets $A,B,C$ with $\mu(A)\mu(B)>0$ such that $C\supset (1-\la)A\star\la B$.

Under the same initial assumptions of Theorem \ref{t:BM_ProdMetSpa},
and following its proof, we may assert that if $(X,\di_X,\mu_X)$, $(Y,\di_Y,\mu_Y)$ satisfy $\overline{\mathrm{BM}}(1,q)$ and $\mathrm{BM}(p,q)$, respectively, for some $0\neq p\geq-1$ and $q\in(1,\infty)$, then $(X\times Y,\di_{X\times Y},\mu_{X\times Y})$ satisfies $\mathrm{BM}\bigl(1/(1+p^{-1}),q\bigr)$. We notice that, since $(1-\lambda)^q+\lambda^q\neq1$, here the case $p=0$ makes no sense.

A not so trivial generalization of Theorem~\ref{t:BM_ProdMetSpa}, pointed out to the authors by Luca Rizzi, is the following.

\begin{teor}\label{t:BM_ProdMetSpa_gen}
Let $(X,\di_X,\mu_X)$, $(Y,\di_Y,\mu_Y)$ be metric measure spaces where $\mu_X$ is $\sigma$-finite, $\mu_Y$ is locally finite and $\mu_{X\times Y}$ is a Radon measure. Let $n,N>0$.

Assume that $(X,\di_X,\mu_X)$, $(Y,\di_Y,\mu_Y)$ satisfy $\overline{\mathrm{BM}}(1)$ and $\mathrm{BM}(1/n,N/n)$, respectively.
Then $(X\times Y,\di_{X\times Y},\mu_{X\times Y})$ satisfies $\mathrm{BM}\bigl(1/(n+1),(N+1)/(n+1))$.
\end{teor}

\begin{proof}
The proof follows closely the one of Theorem~\ref{t:BM_ProdMetSpa}. By similar arguments it is enough to show that the desired Brunn-Minkowski inequality holds for compact subsets $A,B$ in $X\times Y$. We shall denote the product measure $\mu_{X\times Y}$ simply by $\mu$ and let $p=1/n$, $q=N/n$.

We consider the non-negative functions $f,g,h:X\longrightarrow\R_{\geq0}$ given by
\[f(t)=\frac{\mu_{X}(A(t))}{\enorm{\mu_{X}(A(\cdot))}_\infty}, \ g(t)=\frac{\mu_{X}(B(t))}{\enorm{\mu_{X}(B(\cdot))}_\infty},
\ h(t)=\frac{\mu_{X}(C(t))}{C_{p,q}},\]
where
\begin{equation*}
\begin{split}
C_{p,q}&=M_{p,q}\big(\enorm{\mu_{X}(A(\cdot))}_\infty, \enorm{\mu_{X}(B(\cdot))}_\infty,\lambda\big)\\
&=\bigl((1-\lambda)^q\enorm{\mu_{X}(A(\cdot))}_\infty^p+\lambda^q\enorm{\mu_{X}(B(\cdot))}_\infty^p\bigr)^{1/p}.
\end{split}
\end{equation*}
The above functions are well-defined and
\[\sup_{t\in X}f(t)=\sup_{t\in X}g(t)=1.\]

For any pair of points $s,r\in X$, and any $t_\la\in (1-\la)\{s\}\star\la\{r\}$, we have
\begin{equation}\label{e:condit_hfg_gen}
h(t_\lambda)\geq\min\{f(s),g(r)\}.
\end{equation}
To check the validity of \eqref{e:condit_hfg_gen} is enough to consider the case $\mu_{X}(A(s))\mu_{X}(B(r))>0$. Hence $A(s), B(r)$ are non-empty, the inclusion
\[
C(t_\la)\supset\bigl((1-\lambda)A\star\lambda B\bigr)(t_\lambda)\supset(1-\lambda)A(s)\star\lambda B(r)
\]
trivially holds, and we have
\begin{equation*}
\begin{split}
C_{p,q}h(t_\la)=\mu_{X}(C(t_\la))
&\geq\mu_{X}\bigl((1-\lambda)A(s)\star\lambda B(r)\bigr)\\
&\ge M_{p,q}\big(\mu_{X}(A(s)),\mu_{X}(B(r)),\la\big)
\\
&\ge C_{p,q}\min\{f(s),g(r)\}.
\end{split}
\end{equation*}
Hence \eqref{e:condit_hfg_gen} follows.

The definition of $f$ and $g$ implies that the sets
\[
\{f\ge t\}=\bigl\{x\in X:\,f(x)\geq t\bigr\},\quad \{g\ge t\}=\bigl\{x\in X:\,g(x)\geq t\bigr\}
\]
are non-empty for any $t\in [0,1)$. Now \eqref{e:condit_hfg_gen} trivially implies
\[
\{h\ge t\}\supset(1-\lambda)\, \{f\ge t\}\star\lambda\, \{g\ge t\}
\]
and, since $X$ satisfies $\overline{\mathrm{BM}}(1)$, we have
\begin{equation*}
\mu_{X}(\{h\ge t\})\ge (1-\la)\mu_{X}(\{f\ge t\})+\la\mu_{X}(\{g\ge t\}
\end{equation*}
for any $t\in [0,1)$.

Following similar steps to those described at the end of the proof of Theorem~\ref{t:linBM_ProdMetSpa}, we get
\begin{equation*}
\begin{split}
\mu(C)&=C_{p,q}\int_X h(x)\,\dlat \mu_X(x)\\
&\geq C_{p,q}\left((1-\lambda)\int_X f(x)\,\dlat \mu_X(x) + \lambda\int_X g(x)\,\dlat \mu_X(x)\right)\\
&=C_{p,q}\left((1-\lambda)\frac{\mu(A)}{\enorm{\mu_{X}(A(\cdot))}_\infty}+
\lambda\frac{\mu(B)}{\enorm{\mu_{X}(B(\cdot))}_\infty}\right).
\end{split}
\end{equation*}
We apply the reverse H\"older inequality
\[
(1-\la)a_1b_1+\la a_2 b_2\ge \big((1-\la)a_1^{-p}+\la a_2^{-p}\big)^{-1/p}
\big((1-\la)b_1^{p^\prime}+\la b_2^{p^\prime}\big)^{1/p^\prime},
\]
see \cite[Theorem~1, p.~178]{Bu}, where $p^\prime=1/(1+p^{-1})=(n+1)^{-1}$ and
\begin{align*}
a_1&=(1-\la)^{-(N-n)}|\mu_{Y}(A(\cdot))|_\infty^{-1},
\\
a_2&=(1-\la)^{-(N-n)}|\mu_{Y}(B(\cdot))|_\infty^{-1},
\\
b_1&=(1-\la)^{N-n}\mu(A),
\\
b_2&=\la^{N-n}\mu(B).
\end{align*}
This implies $\mu(C)\ge M_{p,q}(\mu(A),\mu(B),\la)$ and completes the proof.
\end{proof}

Theorem~\ref{t:BM_ProdMetSpa_gen} can be applied to the product of one or several copies of the real line with a Carnot group. The validity of the above Brunn-Minkowski inequality for corank $1$ Carnot groups has been recently established by Balogh et al.,~see Theorem~4.2 (ii) in \cite{balogh2}.

From the proof of Theorem~\ref{t:BM_ProdMetSpa}, we may obtain the following corollary, which will be very useful in order to obtain some Brunn-Minkowski inequa\-li\-ties for certain subfamilies of sets as, for instance, Theorem \ref{t:BM_prod_quasi-conc} shows.
\begin{corollary}\label{c1}
Let $(X,\di_X,\mu_X)$, $(Y,\di_Y,\mu_Y)$ be metric measure spaces for which there exist certain families $\mathcal{F}_X\subset \mathcal{P}(X)$, $\mathcal{F}_Y\subset \mathcal{P}(Y)$ that satisfy  $\overline{\mathrm{BM}}(1)$ and $\mathrm{BM}(p)$, respectively, where $p\geq-1$ and $\mu_X, \mu_Y$ are $\sigma$-finite.
Let $A,B\subset X\times Y$ be measurable sets such that $(1-\lambda)A\star\lambda B$ is so for $\lambda\in(0,1)$. If moreover $A,B$ satisfy:
\begin{enumerate}
  \item The non-negative functions $f,g:X\longrightarrow\R_{\geq0}$ given by
   \[f(t)=\frac{\mu_Y(A(t))}{\enorm{\mu_Y(A(\cdot))}_\infty}, \quad g(t)=\frac{\mu_Y(B(t))}{\enorm{\mu_Y(B(\cdot))}_\infty}\]
   are well-defined, i.e., $0<\enorm{\mu_Y(A(\cdot))}_\infty,\enorm{\mu_Y(B(\cdot))}_\infty <+\infty$,

   \smallskip

   \item $A(t), B(t)\in\mathcal{F}_Y$ for all $t\in X$,

  \smallskip

  \item $\{x\in X:f(x)\geq t\},\{x\in X:g(x)\geq t\}\in\mathcal{F}_X$ for any $0<t<1$ (or a.e.),
\end{enumerate}
then
\begin{equation*}
\mu_{X\times Y}\bigl((1-\lambda)A\star\lambda B\bigr)
\geq M_{1/(1+p^{-1})}\big(\mu_{X\times Y}(A),\mu_{X\times Y}(B),\la\big)
\end{equation*}
\end{corollary}

\begin{remark}
We would like to point out that, following the ideas of the proof of Theorem \ref{t:BM_ProdMetSpa}, we cannot expect to exchange the role of the linear Brunn-Minkowski inequality $\overline{\mathrm{BM}}(1)$ by a different one $\overline{\mathrm{BM}}(p)$, $p\neq 1$. Indeed, if that was the case, it would be possible to get an enhanced version of the classical Brunn-Minkowski inequality when assuming a common maximal volume section through parallel planes (of dimension smaller than $n-1$), a fact that is known to be not true (see \cite[Section~2]{HCYN2}).
\end{remark}

From the proof of Theorem \ref{t:BM_ProdMetSpa}, we may also assert the following:
\begin{corollary}\label{c:PL}
Let $(X,\di_X,\mu_X)$, $(Y,\di_Y,\mu_Y)$ be metric measure spaces
where $\mu_X$ is $\sigma$-finite, $\mu_Y$ is locally finite and $\mu_{X\times Y}$ is Radon.

If the spaces $(X,\di_X,\mu_X)$, $(Y,\di_Y,\mu_Y)$ satisfy $\mathrm{PL}$ and $\mathrm{BM}(0)$, respectively,
then $(X\times Y,\di_{X\times Y},\mu_{X\times Y})$ satisfies $\mathrm{BM}(0)$.
\end{corollary}

\medskip

In the same way as we use ``simple'' spaces that satisfy some Brunn-Minkowski inequality in order to obtain others on ``more involved'' ones, we will take advantage of spaces that satisfy the Pr\'ekopa-Leindler inequality, or more generally the so-called Borell-Brascamp-Lieb inequality, which is a generalization of Theorem \ref{t:PrekopaLeindler} for $p$-th means (see \cite{BL}, \cite{Borell} and also \cite{G} for a detailed presentation). We collect it here for the sake of completeness.
\begin{theorem}[Borell-Brascamp-Lieb inequality]\label{t:BBL}
Let $\lambda\in(0,1)$, $-1/n\leq p\leq\infty$ and let $f,g,h:\R^n\longrightarrow\R_{\geq0}$ be non-negative measurable
functions such that, for any $x,y\in\R^n$,
\begin{equation*}
h\bigl((1-\lambda) x+\lambda y\bigr)\geq M_{p}\bigl(f(x), g(y),\lambda\bigr).
\end{equation*}
Then
\begin{equation*}\label{e:BBL}
\int_{\R^n}h\,\dlat x\geq
M_{p/(np+1)}\left(\int_{\R^n}f\,\dlat x,\,\int_{\R^n}g\,\dlat x,\,\lambda\right).
\end{equation*}
\end{theorem}

Now, in the spirit of Definition \ref{d:1} and with the above-mentioned goal in mind, we give the following definition.
\begin{definition}
We say that $(X,\di_X,\mu_X)$ satisfies the Borell-Brascamp-Lieb inequality with respect to $m=m(X)>0$ for the value $p\geq-1/m$, $\mathrm{BBL}(p,m)$ for short, if
for any $\lambda\in (0,1)$ and non-negative ($\mu_X$)-measurable
functions $f,g,h:X\longrightarrow\R_{\geq0}$ such that
\begin{equation}\label{e:BBLCondicion}
h(z)\geq M_p(f(x),g(y),\lambda),
\end{equation}
for all $x,y\in X$, and $z\in(1-\lambda)\{x\}\star\lambda\{y\}$,
then
\begin{equation*}
\int_{X}h\,\dlat\mu_X\geq
M_q\left(\int_{X}f\,\dlat\mu_X,\int_{X}g\,\dlat\mu_X,\lambda\right),
\end{equation*}
where $q=q(p,m)=\frac{p}{mp+1}$.

In the same way, a function $\phi:X\longrightarrow\R_{\geq0}$ is said to be $p$-concave, for $p\in\R\cup\{\pm\infty\}$, if
\begin{equation}\label{e:p-concave-Condicion}
\phi(z)\geq M_p(\phi(x),\phi(y),\lambda),
\end{equation}
for all $x,y\in X$, $z\in(1-\lambda)\{x\}\star\lambda\{y\}$, and any $\lambda\in(0,1)$.
\end{definition}

Taking a look at the result collected in Example \ref{ex:1.2} (cf.~also Theorem \ref{t:linBM_ProdMetSpa}),
at a first moment, one could think that
\begin{equation}\label{e:linGaussBM}
\gamma_n((1-\lambda)K+\lambda L)\geq(1-\lambda)\gamma_n(K)+\lambda\gamma_n(L)
\end{equation}
holds for convex bodies $K$, $L$ for which
\[K|H=L|H \quad \text{ or } \quad \max_{x\in H^\perp}\gamma_{n-1}\bigl(K\cap(x+H)\bigr)=\max_{x\in
H^\perp}\gamma_{n-1}\bigl(L\cap(x+H)\bigr),
\]
for a certain $H\in\L^n_{n-1}$. Considering once again $K=\bola$, $L=x_0+\bola$ (where $x_0$ is large enough),
one may observe that the above linear version of the Gaussian Brunn-Minkowski inequality does not hold under these assumptions. However, if one replaces the condition on a common maximal $(n-1)$-dimensional (Gaussian) measure section with a common maximal $(n-1)$-dimensional ``marginal measure'' section, one can show that \eqref{e:linGaussBM} holds. This is the content of the following theorem, in the more general setting of metric measure spaces (for an analytic version of it in $\R^n$ we refer the reader to \cite[Theorem~3.2]{DaUr} and the references therein). As a particular case, for $p=0$, Theorem \ref{t:BM_lin1_MetSpa} is obtained.
\begin{teor}\label{t:BM_lin2_MetSpa}
Let $(X,\di_X,\mu_X)$, $(Y,\di_Y,\mu_Y)$ be metric measure spaces which satisfy $\overline{\mathrm{BM}}(1)$ and $\mathrm{BBL}(p,m)$, respectively, with some $m>0$. Let $\mu$ be the measure on $X\times Y$ given by $\dlat\mu(x,y)=\phi(x,y)\dlat\mu_{X\times Y}$ where $\phi$ is a $p$-concave function, and $\mu_X, \mu_Y$ are $\sigma$-finite.
Let $\lambda\in(0,1)$ and let $A,B\subset X\times Y$ be non-empty measurable sets such that
$(1-\lambda)A\star\lambda B$ is also measurable and so that
\begin{equation*}
\sup_{x\in X}\int_Y\chi_{_A}(x,y)\phi(x,y)\,\dlat\mu_Y=\sup_{x\in X}\int_Y\chi_{_B}(x,y)\phi(x,y)\,\dlat\mu_Y.
\end{equation*}
Then
\begin{equation*}
\mu\bigl((1-\lambda)A\star\lambda B\bigr)\geq(1-\lambda)\mu(A)+\lambda\,\mu(B).
\end{equation*}
\end{teor}
\begin{proof}
Let $f,g,h:X\times Y\longrightarrow\R_{\geq0}$ be the functions given by
$f=\chi_{_A}\phi$, $g=\chi_{_B}\phi$ and $h=\chi_{_{(1-\lambda)A\star\lambda B}}\phi$.

Given $x_1,y_1\in X$ and $z_1\in(1-\lambda)\{x_1\}\star\lambda\{y_1\}$, let
$f_{x_1},g_{y_1},h_{z_1}:Y\longrightarrow\R_{\geq0}$ and $F,G,H:X\longrightarrow\R_{\geq0}$ given by
$f_{x_1}(\cdot)=f(x_1,\cdot)$, $g_{y_1}(\cdot)=g(y_1,\cdot)$, $h_{z_1}(\cdot)=h(z_1,\cdot)$,
$F(x_1)=\int_Y f_{x_1}\dlat\mu_Y$, $G(y_1)=\int_Y g_{y_1}\dlat\mu_Y$, $H(z_1)=\int_Y h_{z_1}\dlat\mu_Y$.

From $\eqref{e:p-concave-Condicion}$, we have
\begin{equation*}
h_{z_1}(z_2)\geq M_p(f_{x_1}(x_2),g_{y_1}(y_2),\lambda),
\end{equation*}
for all $x_2,y_2\in Y$, and $z_2\in(1-\lambda)\{x_2\}\star\lambda\{y_2\}$. Thus, by the BBL inequality in $Y$ we obtain
\begin{equation}\label{e:FGH}
H(z_1)\geq
M_{q}\left(F(x_1),G(y_1),\lambda\right),
\end{equation}
where $q=q(p,m)=\frac{p}{mp+1}$.

By hypothesis $\sup_{x\in X}F(x)=\sup_{x\in X}G(x)=:\alpha$ and, from \eqref{e:FGH},
\begin{equation*}
\{x\in X: H(x)\geq t\}\supset(1-\lambda)\{x\in X:F(x)\geq t\}\star\lambda\{x\in X: G(x)\geq t\}
\end{equation*}
for all $0<t<\alpha$.
The proof now concludes as at the end of the proof of Theorem \ref{t:linBM_ProdMetSpa}.
\end{proof}

As a consequence of the above result, and since the Gaussian density is log-concave, we have that \eqref{e:linGaussBM} holds under the assumption of a common maximal marginal measure section. This fact will be collected in Theorem \ref{t:BM_lineal_med_2}.

We finish this section by collecting some results for product spaces in the spirit of the precedent ones,
and which will be useful in order to derive some further inequalities, as it will be shown in Section \ref{s:applications}.

\begin{teor}
Let $m_1, m_2>0$ and $p\geq-1/(m_1+m_2)$. If $(Y,\di_Y,\mu_Y)$ satisfies $\mathrm{BBL}(p,m_2)$ and $(X,\di_X,\mu_X)$ satisfies $\mathrm{BBL}(q,m_1)$ where $q=q(p,m_2)$, then $(X\times Y,\di_{X\times Y},\mu_{X\times Y})$ satisfies $\mathrm{BBL}(p,m_1+m_2)$.
\end{teor}
\begin{proof}
Let $\lambda\in(0,1)$ and let $f,g,h:X\times Y\longrightarrow\R_{\geq0}$ be measurable functions satisfying $\eqref{e:BBLCondicion}$ for all $x,y\in X\times Y$, and $z\in(1-\lambda)\{x\}\star\lambda\{y\}$.

Arguing in the same way as in the above proof (and using the same notation), by the BBL inequality in $Y$ (we notice that $p\geq-1/(m_1+m_2)\geq-1/m_2$) we obtain
\begin{equation*}
H(z_1)\geq
M_{q}\left(F(x_1),G(y_1),\lambda\right),
\end{equation*}
where $q=q(p,m_2)=\frac{p}{m_2p+1}\geq-\frac{1}{m_1}$.

Now, from the BBL inequality in $X$ we get
\begin{equation*}
\begin{split}
\int_{X\times Y}h\,\dlat\mu_{X\times Y}=\int_{X}H\,\dlat\mu_X
&\geq M_{\tilde{q}}\left(\int_{X}F\,\dlat\mu_X,\int_{X}G\,\dlat\mu_X,\lambda\right)\\
&=M_{\tilde{q}}\left(\int_{X\times Y}f\,\dlat\mu_{X\times Y},\int_{X\times Y}g\,\dlat\mu_{X\times Y},\lambda\right),
\end{split}
\end{equation*}
where $\tilde{q}=\tilde{q}(q,m_1)=\frac{q}{m_1q+1}=\frac{p}{(m_1+m_2)p+1}$. It finishes the proof.
\end{proof}

As straightforward consequences of the above result we get the following corollaries.
\begin{corollary}\label{c:ndimBBL}
Let $m>0$ and $p\geq-1/(nm)$, where $n\in\Z_{>0}$. If $(X,d_X,\mu_X)$ satisfies $\mathrm{BBL}(p/(rmp+1),m)$, for $r=0,1,\dots,n-1$, then $(X^n,d_{X^n},\mu_{X^n})$ satisfies $\mathrm{BBL}(p,mn)$.
\end{corollary}

\begin{corollary}
Let $(X,\di_X,\mu_X)$ and $(Y,\di_Y,\mu_Y)$ be metric measure spaces which satisfy $\mathrm{PL}$.
Then $(X\times Y,\di_{X\times Y},\mu_{X\times Y})$ also satisfies $\mathrm{PL}$. In particular, the space
$(X^n,d_{X^n},\mu_{X^n})$ satisfies $\mathrm{PL}$.
\end{corollary}

\section{Brunn-Minkowski inequalities for product measures on $\R^n$}\label{s:BM_Rn}

We start this section by showing Theorem \ref{t:BM_prod_quasi-conc} which, as mentioned in the introduction, allows us to assert that the Gaussian Brunn-Minkowski inequality \eqref{e:GaussBM} holds for the more general case of \emph{weakly unconditional sets}.
We present a simple proof of it, based on a direct application of Corollary \ref{c1} and Example \ref{ex:2}.

\begin{proof}[Proof of Theorem \ref{t:BM_prod_quasi-conc}]
Without loss of generality, we may assume that $\mu$ is not the zero measure.
Hence, since $\phi_i$ is positively decreasing for $i=1,\dots,n$, the measure $\mu_i$ is strictly positive around the origin and locally finite, and we get that the product measure $\mu$ is a Radon measure. Thus, by the proof of Proposition~\ref{p:BMzeromeas_sets}, since $0\in A\cap B$, we may assume that the sets $A, B$ satisfy $\mu(A)\mu(B)>0$.
Furthermore, as at the beginning of the proof of Theorem \ref{t:BM_ProdMetSpa}, we may assume that $A, B$ are compact.

The theorem now follows from recursively applying Corollary \ref{c1} with $X=\R$ and $Y=\R^m$, $m=1,\dots,n-1$, taking the families of weakly unconditional sets  $\mathcal{F}_X$, $\mathcal{F}_Y$ in $\R$ and $\R^m$. The one-dimensional case was proved in Example~\ref{ex:2}. We notice that conditions ii) and iii) of Corollary \ref{c1} are satisfied because of the definition of weakly unconditional sets.
\end{proof}

\begin{remark}\label{r:GaussBM}
We would like to point out some facts regarding the necessity of the conditions in Theorem \ref{t:BM_prod_quasi-conc} as well as the consequences of this result in relation to the Gaussian Brunn-Minkowski inequality \eqref{e:GaussBM}.

\smallskip

i) Taking into account the above-mentioned negative result from \cite{NaTk} one may think that the weakly unconditional case is the strongest one that may be expected regarding Conjecture \ref{conjecture}.
Indeed, since both sets in \eqref{e:ContraejGBM} contain the origin and furthermore their projection onto the $y$-axis, the sole ``missing points'' which make them impossible for the sets to be weakly unconditional are those belonging to the $x$-axis; however, the Gaussian Brunn-Minkowski inequality is not true for such sets.

\smallskip

ii) The assumption on the measure in Theorem \ref{t:BM_prod_quasi-conc} of being a product measure is needed, as it is shown in \cite[Example~1]{LiMaNaZv} via the density $\varphi(x)=\frac{1}{2}\chi_{_{2C}}(x)+\frac{1}{2}\chi_{_{C}}(x)$, where $C$ is the square
$\{(x,y)\in\R^2: \,\enorm{x}\leq1, \enorm{y}\leq1\}$.

\smallskip

iii) Furthermore, the assumption on the density functions in Theorem \ref{t:BM_prod_quasi-conc} of being positively decreasing is needed, as Proposition \ref{p:linBMimplies_quasi} shows.
\end{remark}

As a straightforward consequence of Theorem \ref{t:BM_prod_quasi-conc} together with the fact that $(r+s)C=rC+sC$ for any convex set $C\subset\R^n$ and any $r,s\geq0$, we get the following result.
\begin{corollary}\label{c:1/n-concavity}
Let $\mu=\mu_1\times\dots\times\mu_n$ be a product measure on $\R^n$ such that $\mu_i$ is the
measure given by $\dlat\mu_i(x)=\phi_i(x)\,\dlat x$,
where $\phi_i:\R\longrightarrow\R_{\geq0}$ is a positively decreasing function, $i=1,\dots,n$.

Let $\emptyset\neq A,B\subset\R^n$ be weakly unconditional convex sets. Then
the functions $t\mapsto\mu(A+tB)^{1/n}$ and $t\mapsto\mu(tA)^{1/n}$ are concave on $[0,+\infty)$.
\end{corollary}

\medskip

As it occurs in the Euclidean setting, we can deduce an isoperimetric type inequality as a consequence of \eqref{e:BM_prod_quasi-conc}. To this aim, we will introduce some notation.
Let
\[\W^{\mu}_1(A;B)=\V^{\mu}(A[n-1],B[1])=\frac{1}{n}\liminf_{t\to0^+}\frac{\mu(A+tB)-\mu(A)}{t}\]
be the first \emph{quermassintegral} of $A$ with respect to the set $B$ associated to the measure $\mu$. Here we are assuming that $A$ and $B$ are measurable sets such that $A+tB$ is so.

In a similar way we may define
\[\mu^+(A)=\liminf_{t\to0^+}\frac{\mu(A+t\bola)-\mu(A)}{t},\]
the surface area measure associated to $\mu$. Clearly, $\mu^+(A)=n\W^{\mu}_1(A;\bola)$.

Moreover, and following the notation of \cite{LiMaNaZv}, we will write \[M_\mu(A)=n\mu(A)-\deriv{-}{1}\mu(tA),\] provided that the pair $(A,\mu)$ is so that the above (left) derivative exists.
Clearly, $M_{\vol}(A)=0$ for any measurable set $A$ and thus this functional does not appear in the classical isoperimetric inequality. For more information about the role of this functional in the literature, we refer the reader to \cite{LiMaNaZv} and the references therein.

With this notation, we get an isoperimetric type inequality in Theorem~\ref{t:isop}. This result was previously obtained in \cite[Corollary~4]{LiMaNaZv}, without the equality case, in the setting of Theorem \ref{t:LiMaNaZv}.
The main idea of the proof we present here goes back to the classical proof of the Minkowski first inequality that can be found in \cite[Theorem~7.2.1]{Sch2}.
\begin{teor}\label{t:isop}
Let $\mu=\mu_1\times\dots\times\mu_n$ be a product measure on $\R^n$ such that $\mu_i$ is the
measure given by $\dlat\mu_i(x)=\phi_i(x)\,\dlat x$,
where $\phi_i:\R\longrightarrow\R_{\geq0}$ is a positively decreasing function, $i=1,\dots,n$.

Let $A,B\subset\R^n$ be non-empty weakly unconditional convex sets such that
$\mu\bigl((1-\lambda)A+\lambda B\bigr)<+\infty$ for all $\lambda\in[0,1]$. Then
\begin{equation}\label{e:isop1}
\W^{\mu}_1(A;B)+\frac{1}{n}M_\mu(A)\geq\mu(A)^{1-1/n}\mu(B)^{1/n},
\end{equation}
with equality if $A=B$.

In particular, for any $r>0$,
\begin{equation}\label{e:isop2}
r\mu^+(A)+ M_\mu(A)\geq n\mu(A)^{1-1/n}\mu(r\bola)^{1/n},
\end{equation}
with equality if $A=r\bola$.
\end{teor}
\begin{proof}
Let $f:[0,1]\longrightarrow\R_{\geq0}$ be the function given by
\begin{equation}\label{e:auxfunction_proofisopineq}
 f(t)=\mu\bigl((1-t)A+tB\bigr)^{1/n}-\bigl((1-t)\mu(A)^{1/n}+t\mu(B)^{1/n}\bigr).
\end{equation}
By Theorem \ref{t:BM_prod_quasi-conc} $f$ is a concave function (we notice that the fact of being a weakly unconditional set is closed under convex combinations) satisfying $f(0)=f(1)=0$. Thus, the right derivative of $f$ at $t=0$ exists
(cf.~\cite[Theorem~23.1]{Ro}) and furthermore
\begin{equation}\label{e:f'+(0)}
\deriv{+}{0}f(t)\geq0
\end{equation}
with equality if and only if $f(t)=0$ for all $t\in[0,1]$, i.e., if and only if \eqref{e:BM_prod_quasi-conc}
holds with equality for all $t\in[0,1]$.

Now, since
\begin{equation*}
\deriv{+}{0}f(t)=\frac{1}{n}\mu(A)^{1/n-1}\deriv{+}{0}\mu\bigl((1-t)A+tB\bigr)+\mu(A)^{1/n}-\mu(B)^{1/n},
\end{equation*}
we just must compute the right derivative at $0$ of $\mu\bigl((1-t)A+tB\bigr)$ to conclude the proof.

To this end, we notice that, by Corollary \ref{c:1/n-concavity}, the one-sided derivatives of $\mu(tA)$ and $\mu(A+tB)$ at $t=1$ and $t=0$, respectively, exist. Hence, writing $g(r,s)=\mu\bigl(r(A+sB)\bigr)$, we have
\begin{equation*}
\begin{split}
\deriv{+}{0}\mu\bigl((1-t)A+tB\bigr)&=\deriv{+}{0} g\left(1-t,\frac{t}{1-t}\right)\\
=-\deriv{-}{1}\mu(tA)+\deriv{+}{0}\mu(A+tB)&=M_\mu(A)-n\mu(A)+ n\W^{\mu}_1(A;B),
\end{split}
\end{equation*}
and thus
\begin{equation*}
\deriv{+}{0}f(t)=\frac{1}{n}\mu(A)^{1/n-1}\bigl(M_\mu(A)-n\mu(A)+ n\W^{\mu}_1(A;B)\bigr)+\mu(A)^{1/n}-\mu(B)^{1/n}.
\end{equation*}

Now, the latter identity, together with \eqref{e:f'+(0)}, gives \eqref{e:isop1}. Finally, the assertion about the equality condition comes from the characterization of the equality case in \eqref{e:f'+(0)}, whereas \eqref{e:isop2} is just \eqref{e:isop1} for $B=r\bola$.
\end{proof}

\begin{remark}
We notice that the assumption of convexity in the above result is needed to assure both the existence of the one-sided derivatives of $\mu(tA)$ and $\mu(A+tB)$ at $t=1$ and $t=0$, respectively, and the concavity of the function $f$ defined in \eqref{e:auxfunction_proofisopineq}.

However, assuming that the derivatives
\begin{equation}\label{e:deriv}
\deriv{\,\,}{1}\mu(tA), \quad  \deriv{+}{0}\mu(A+tr\bola)
\end{equation}
exist, one may obtain the same inequalities \eqref{e:isop1} and \eqref{e:isop2} just by applying Theorem \ref{t:BM_prod_quasi-conc} and ``differentiating'' in both sides (see the proof of \cite[Corollary~4]{LiMaNaZv}). Thus, by the above result, and assuming that the derivative of $\mu(tr\bola)$ at $t=1$ exists, we may assert the following:

\emph{Euclidean balls} $r\bola$ \emph{minimize the functional} $r\mu^+\bigl(\overline{A}\bigr)+ M_\mu\bigl(\overline{A}\bigr)$ \emph{among all sets} $A$ in $\R^n$ \emph{with predetermined measure} $\mu(A)=\mu(r\bola)$ (provided that \eqref{e:deriv} exist for their weakly unconditional hull $\overline{A}$).
\end{remark}

\begin{remark}
When $\mu$ is the classical Lebesgue measure $\vol$ in $\R^n$, inequality \eqref{e:isop1} becomes the classical Minkowski first inequality since $M_\vol=0$. From this, the classical Euclidean isoperimetric inequality is easily obtained. We would like to point out in this remark that the isoperimetric inequality may also be obtained from the linear Brunn-Minkowski inequality in Example~\ref{e:BMproye_vol} for sets having a maximal section of equal area as follows:

Assume that $E\subset\R^n=\R\times\R^{n-1}$ is a measurable set of finite volume with
\[
\sup_{t\in\R}\vol_{n-1}(E(t))<+\infty,
\]
where $E(t)=\{x\in\R^{n-1}: (t,x)\in E\}$. Take $r>0$ so that
\[
\sup_{t\in\R} \vol_{n-1}(E(t))=\sup_{t\in\R} \vol_{n-1}(r\B(t)),
\]
where $\B=\B_n\subset\R^n$ is the closed unit ball. Theorem~3.1 then implies
\[
\vol_n((1-\la)E+\la(r\B))\ge (1-\la)\vol_n(E)+\la\vol_n(r\B).
\]
Substracting $(1-\la)^n\vol_n(E)$ from both sides and dividing by $\la$, we obtain
\begin{multline*}
(1-\la)^{n-1}\frac{\vol_n(E+\tfrac{\la}{1-\la}(r\B))-\vol_n(E)}{\tfrac{\la}{1-\la}}
\\
\ge
(1-\la)\frac{(1-(1-\la)^{n-1})\vol_n(E)}{\la}+\vol_n(r\B).
\end{multline*}
Taking $\liminf$ when $\la$ goes to $0$, since $nW_1^\vol(E,r\B)=r\mu^+(E)$, we get
\[
\mu^+(E)-\frac{(n-1)}{r}\vol_n(E)\ge r^{n-1}\vol_n(\B).
\]
As $\mu^+(\B)=n\vol_n(\B)$, we obtain
\begin{equation}
\label{eq:1stineq}
\mu^+(E)\ge \mu^+(r\B)+\frac{(n-1)}{r}\big(\vol_n(E)-\vol_n(r\B)\big).
\end{equation}

Now define the function
\[
f(r):=\mu^+(r\B)+\frac{(n-1)}{r}\big(\vol_n(E)-\vol_n(r\B)\big).
\]
Its derivative is given by
\[
f'(r)=\frac{(n-1)}{r^2}\big(\vol_n(r\B)-\vol_n(E)\big).
\]
Hence for the unique $r_0>0$ such that $\vol_n(r_0\B)=\vol_n(E)$ we have $f'(r_0)=0$. Moreover, $f'(r)>0$ when $r>r_0$ and $f'(r)<0$ for $r<r_0$. This implies that $r_0$ is a global minimum for $f$. From \eqref{eq:1stineq} we get
\begin{equation*}\label{eq:Euclisop}
\mu^+(E)\ge f(r)\ge f(r_0)=\mu^+(r_0\B).
\end{equation*}
This is the Euclidean isoperimetric inequality.

We stress the fact that only the linear isoperimetric inequality in the Euclidean space is necessary to obtain the classical isoperimetric inequality.
\end{remark}

Although, as we have commented in Section \ref{s:main}, in general it is not enough to consider sets with a common maximal measure section (through parallel hyperplanes) in order to get a linear Brunn-Minkowski inequality, they will not be so far from being convenient sets for this goal: it is sufficient to impose further that this maximum is attained at the `central section'. This is the content of the following result.
\begin{teor}\label{t:BM_lineal_med_1}
Let $\mu_1^{(1)}$ be the measure on $\R$ given by $\dlat\mu_1^{(1)}(x)=\phi_1(x)\,\dlat x$, where $\phi_1:\R\longrightarrow\R_{\geq0}$ is a positively decreasing function, and let $\mu_2^{(n-1)}$ be the
measure on $\R^{n-1}$ given by $\dlat\mu_2^{(n-1)}(x)=\phi_2(x)\,\dlat x$,
where $\phi_2:\R^{n-1}\longrightarrow\R_{\geq0}$ is a $q$-concave function, with $-1/(n-1)\leq q\leq \infty$.
Consider the product measure $\mu=\mu_1^{(1)}\times\mu_2^{(n-1)}$ on $\R^n$.

Let $\lambda\in(0,1)$ and let $A,B\subset\R^n$ be non-empty measurable sets such that
$(1-\lambda)A+\lambda B$ is also measurable and so that
\begin{equation*}
\mu_2^{(n-1)}(A(0))=\sup_{x\in\R}\mu_2^{(n-1)}(A(x))=\sup_{x\in\R}\mu_2^{(n-1)}(B(x))=\mu_2^{(n-1)}(B(0)).
\end{equation*}
Then
\begin{equation*}
\mu\bigl((1-\lambda)A+\lambda B\bigr)\geq(1-\lambda)\mu(A)+\lambda\mu(B).
\end{equation*}
\end{teor}
\begin{proof}
First we notice that, by Theorem \ref{t:BBL}, $(\R^{n-1},\enorm{\,\cdot\,},\mu_2^{(n-1)})$ satisfies $\mathrm{BM}(p)$ for $p=q/((n-1)q+1)$. Indeed, it is enough to consider the functions $f=\chi_{_{C}}\phi_2$, $g=\chi_{_{D}}\phi_2$ and $h=\chi_{_{E}}\phi_2$ for any measurable sets $C,D,E\subset\R^{n-1}$ such that $E\supset (1-\lambda) C+\lambda D$, and apply Theorem \ref{t:BBL}.

Therefore, now the proof is an immediate consequence of Corollary \ref{c:linBM_ProdMetSpa} (with $X=\R$, $Y=\R^{n-1}$ and $\mathcal{F}_X$ being the family of sets in $\R$ containing the origin) together with Example \ref{ex:2}.
\end{proof}

Let $\mu$ be the measure on $\R^n$ given by $\dlat\mu(x)=\phi(x)\dlat x$.
Given a hyperplane $H\in\L^n_{n-1}$, we will denote by $\widetilde{\mu}_{n-1}$ the `marginal' of the measure $\mu$
with respect to $H$, i.e., for any $y\in H^\perp$ and $C\subset y+H$,
\begin{equation*}
\widetilde{\mu}_{n-1}(C)=\int_C \phi_{|_{\left(y+H\right)}}\,\dlat x.
\end{equation*}
With this notation, we have the following result, which is a direct consequence of Theorem \ref{t:BM_lin2_MetSpa} together with Theorem \ref{t:BBL}. It can be also obtain from \cite[Theorem~3.2]{DaUr}; we include it here for the sake of completeness.

\begin{teor}\label{t:BM_lineal_med_2}
Let $\mu$ be the measure on $\R^n$ given by $\dlat\mu(x)=\phi(x)\dlat x$, where $\phi$ is a $p$-concave function,
with $-1/n\leq p\leq \infty$.

Let $\lambda\in(0,1)$ and let $A,B\subset\R^n$ be non-empty measurable sets such that
$(1-\lambda)A+\lambda B$ is also measurable and so that
\begin{equation*}
\sup_{y\in H^\perp}\widetilde{\mu}_{n-1}\bigl(A\cap(y+H)\bigr)=
\sup_{y\in H^\perp}\widetilde{\mu}_{n-1}\bigl(B\cap(y+H)\bigr)
\end{equation*}
with respect to a hyperplane $H\in\L^n_{n-1}$.
Then
\begin{equation*}
\mu\bigl((1-\lambda)A+\lambda B\bigr)\geq(1-\lambda)\mu(A)+\lambda\mu(B).
\end{equation*}
\end{teor}

\begin{remark}
Theorems \ref{t:BM_lineal_med_1} and \ref{t:BM_lineal_med_2} provide us with the answer to the desired linear version of the Gaussian Brunn-Minkowski inequality \eqref{e:linGaussBM}. Indeed, it holds for any pair of non-empty measurable sets $A,B\subset\R^n$,
with $(1-\lambda)A+\lambda B$ measurable for $\lambda\in(0,1)$,
for which there exists a hyperplane $H\in\L^n_{n-1}$ such that either
\begin{equation*}
\begin{split}
\gamma_{n-1}(A\cap H)&=\sup_{y\in H^\perp}\gamma_{n-1}\bigl(A\cap(y+H)\bigr)\\
&=\sup_{y\in H^\perp}\gamma_{n-1}\bigl(B\cap(y+H)\bigr)=\gamma_{n-1}(B\cap H)
\end{split}
\end{equation*}
or
\begin{equation*}
\sup_{y\in H^\perp}\widetilde{\gamma}_{n-1}\bigl(A\cap(y+H)\bigr)=
\sup_{y\in H^\perp}\widetilde{\gamma}_{n-1}\bigl(B\cap(y+H)\bigr).
\end{equation*}
\end{remark}

\begin{remark}
If $K\subset\R^n$ is a centrally symmetric convex body, $H$ a linear hyperplane, and $u$ a unit vector orthogonal to $H$, the function $t\mapsto \gamma_{n-1}(K\cap (tu+H))$ is clearly even, and log-concave by the Pr\'ekopa-Leindler inequality, Theorem~\ref{t:PrekopaLeindler}. Hence the maximal $\gamma_{n-1}$-sections corresponds to $t=0$. This implies that, for a pair of centrally symmetric convex bodies with equal $\gamma_{n-1}$-sections in the same direction through the origin, the linear Gaussian Brunn-Minkowski inequality \eqref{e:linGaussBM} holds.
\end{remark}

We finish this section by exploiting the above results, Theorems \ref{t:BM_lineal_med_1} and \ref{t:BM_lineal_med_2},
in order to get a general Brunn-Minkowski type inequality in $\R^n$ when working with measures associated to densities
$\phi:\R^n\longrightarrow\R_{\geq0}$ which are \emph{radially decreasing}, i.e, such that
$\phi(tx)\geq\phi(x)$ for all $x\in\R^n$ and any $t\in[0,1]$. This property, in a sense, will allow us to make up for the lack of homogeneity of the measure. For the sake of brevity, we will present the following result in the setting of Theorem \ref{t:BM_lineal_med_2}, although the same approach could be carried out with Theorem \ref{t:BM_lineal_med_1}.

To this end, we first notice that given two measurable sets $A,B\subset\R^n$ and a hyperplane $H\in\L^n_{n-1}$, and reordering if necessary, we will always have
\begin{equation}\label{e:condic_conseq_lin}
\sup_{y\in H^\perp}\widetilde{\mu}_{n-1}\bigl(A\cap(y+H)\bigr)\geq
\sup_{y\in H^\perp}\widetilde{\mu}_{n-1}\bigl(B\cap(y+H)\bigr).
\end{equation}
Assuming also that $\phi$ is continuous in $\R^n$ and that $B$ is a convex body containing the origin in its interior
then, by continuity arguments, one can always find a certain $t_0\geq1$ such that $A$ and $t_0B$ have a common maximal marginal measure section with respect to the hyperplane $H$. In this case, we may obtain a general Brunn-Minkowski type inequality for $A$ and $B$: this is the content of the following result, which is stated in a slightly more general setting.

\begin{corollary}
Let $\mu$ be the measure on $\R^n$ given by $\dlat\mu(x)=\phi(x)\dlat x$, where $\phi$ is a radially decreasing $p$-concave function, with $-1/n\leq p\leq \infty$.

Let $\lambda\in(0,1)$ and let $A,B\subset\R^n$ be measurable sets such that $(1-\lambda)A+\lambda B$ is also measurable and, without loss of generality, so that \eqref{e:condic_conseq_lin} holds for a given hyperplane $H\in\L^n_{n-1}$.
If the pair $(B,\mu)$ is such that there exists $t_0\geq1$
so that
\begin{equation*}
\sup_{y\in H^\perp}\widetilde{\mu}_{n-1}\bigl(A\cap(y+H)\bigr)=
\sup_{y\in H^\perp}\widetilde{\mu}_{n-1}\bigl((t_0B)\cap(y+H)\bigr),
\end{equation*}
then
\begin{equation*}\label{e:BM_consec_BMlineal}
\mu\bigl((1-\lambda)A+\lambda B\bigr)^{1/n}\geq(1-\lambda)\mu(A)^{1/n}+\lambda\,\frac{\mu\bigl(t_0B\bigr)^{1/n}}{t_0}.
\end{equation*}
\end{corollary}
\begin{proof}
Let $\widetilde{B}=t_0B$. By Theorem \ref{t:BM_lineal_med_2} we clearly have
\begin{equation}\label{e:prueba_BM_conseq_lineal}
\begin{split}
\mu\bigl((1-\overline{\lambda})A+\overline{\lambda}\, \widetilde{B}\bigr)^{1/n} &\geq\left((1-\overline{\lambda})\mu(A)+\overline{\lambda}\mu\bigl(\widetilde{B}\bigr)\right)^{1/n}\\
&\geq (1-\overline{\lambda})\mu(A)^{1/n}+\overline{\lambda}\mu\bigl(\widetilde{B}\bigr)^{1/n}
\end{split}
\end{equation}
for all $\overline{\lambda}\in[0,1]$. Now, we set
\[\overline{\lambda}=\frac{\lambda\frac{1}{t_0}}{(1-\lambda)+\lambda\frac{1}{t_0}}\]
and denote by $D=(1-\lambda)+\lambda/t_0$ (notice that $\overline{\lambda}, D\in(0,1)$). Then, we get
\begin{equation*}
\begin{split}
\frac{1}{D}\mu\bigl((1-\lambda)A+\lambda B\bigr)^{1/n}\geq\mu\left(\frac{(1-\lambda)A+\lambda B}{D}\right)^{1/n}
=\mu\bigl((1-\overline{\lambda})A+\overline{\lambda}\, \widetilde{B}\bigr)^{1/n}\\
\geq (1-\overline{\lambda})\mu(A)^{1/n}+\overline{\lambda}\mu\bigl(\widetilde{B}\bigr)^{1/n}=
\frac{1}{D}\left((1-\lambda)\mu(A)^{1/n}+\lambda\,\frac{\mu\bigl(t_0B\bigr)^{1/n}}{t_0}\right),
\end{split}
\end{equation*}
where in the first inequality above we have applied that $\phi$ is radially decreasing (together with the change of variables theorem) and the second one is just \eqref{e:prueba_BM_conseq_lineal}.
\end{proof}

\section{Other applications}\label{s:applications}

We conclude the paper by deriving some (recently known) Brunn-Min\-kowski inequalities from some results collected in Section \ref{s:main}. To this aim, first we give the following definition.
\begin{definition}
Let $x=(x_1,\dots,x_n),\, y=(y_1,\dots,y_n)\in\R_{>0}^n$ and $r,s>0$. We will denote by $x^r y^s$
the vector
$(x_1^r y_1^s,\dots,x_n^r y_n^s)$. In the same way,  we set
$$\bigl((1-\lambda)x^p+\lambda y^p\bigr)^{1/p}=\Bigl(\bigl((1-\lambda)x_1^p+\lambda y_1^p\bigr)^{1/p},\dots,\bigl((1-\lambda)x_n^p+\lambda y_n^p\bigr)^{1/p}\Bigr),$$ for $p\in(0,1)$.

Thus, given $\emptyset\neq A, B\subset\R_{>0}^n$ and $\lambda\in(0,1)$, we will write
\begin{equation*}
A^{1-\lambda}B^\lambda=\left\{a^{1-\lambda}b^\lambda: \, a\in A,\, b\in B\right\}
\end{equation*}
and
\begin{equation*}
\bigl((1-\lambda)A^{p}+\lambda B^p\bigr)^{1/p}=\left\{\bigl((1-\lambda)a^p+\lambda b^p\bigr)^{1/p}: \, a\in A,\, b\in B\right\}.
\end{equation*}
\end{definition}

\medskip

The classical Minkowski addition of two convex bodies \eqref{e:lincomb} may be extended via the $p$-th mean of two numbers.
More precisely: for $1\leq p\leq\infty$ fixed,
$K,L\subset\R^n$ convex bodies containing the origin and $\lambda\in(0,1)$, there exists a
(unique) convex body $(1-\lambda) K+_p\lambda L$ for which the \emph{support function}
\begin{equation}\label{e:def_p-sum}
h((1-\lambda) K+_p\lambda L,\cdot)^p=(1-\lambda) h(K,\cdot)^p+\lambda h(L,\cdot)^p.
\end{equation}
We recall that $h(K,u)=\max\bigl\{\esc{x,u}:x\in K\bigr\}$,
$u\in\s^{n-1}$, where, as usual, $\s^{n-1}$ denotes the
($n-1$)-dimensional unit sphere of $\R^n$ (for more information about the support function, see e.g. \cite[Section~1.7]{Sch2}).

This $L_p$-mean of convex bodies containing the origin was
introduced and studied by Firey in \cite{F62}) and turned out to be
the starting point for a fruitful theory, the so-called $L_p$\emph{-Brunn-Minkowski theory}
(for further details, we refer the reader to \cite[Subsection~9.1]{Sch2} and the references therein).

Clearly, when $p=1$, formula \eqref{e:def_p-sum} defines the classical
Minkowski mean $(1-\lambda) K+\lambda L$, whereas the case $p=\infty$ gives
\[
(1-\lambda) K+_{\infty}\lambda L=\conv(K\cup L),
\]
where $\conv$ denotes the convex hull of the given set.
Moreover, in \cite[Theorem~1]{F62} it is shown that, for all $1\leq p\leq
q$,
\begin{equation*}
(1-\lambda)K+_p\lambda L\subset(1-\lambda)K+_q\lambda L,
\end{equation*}
$\lambda\in[0,1]$.

In \cite[Theorem~2]{F62}, it was shown the following generalization of the classical Brunn-Minkowski inequality \eqref{e:BM}:
\begin{equation}\label{e: Firey_BM-p}
    \vol\bigl((1-\lambda)K+_p\lambda L\bigr)^{p/n}\geq(1-\lambda)\vol(K)^{p/n}+\lambda\vol(L)^{p/n},
\end{equation}
where $1\leq p\leq\infty$.

Since it is possible to extend the $L_p$-mean \eqref{e:def_p-sum} of two convex bodies $K, L$ containing the origin for the case $0\leq p<1$ via the set
\begin{equation*}\label{e:def p-sum p in [0,1)}
\Bigl\{x\in\R^n:\, \esc{x,u}\leq \bigl((1-\lambda) h(K,u)^p+\lambda h(L,u)^p\bigr)^{1/p} \text{ for all } u\in\s^{n-1}\Bigr\},
\end{equation*}
it is natural to wonder about the possibility of extending also its corresponding Brunn-Minkowski inequality
\eqref{e: Firey_BM-p}. The case $p=0$, i.e.,
\begin{equation*}
    \vol\bigl((1-\lambda)K+_0\lambda L\bigr)\geq\vol(K)^{1-\lambda}\vol(L)^{\lambda},
\end{equation*}
is known in the literature as the \emph{log-Brunn-Minkowski inequality}.
This (conjectured) inequality is, up to our knowledge, still open for arbitrary symmetric convex bodies in dimension $n\geq3$ (it is known to be true in the plane; see \cite{BoLuYaZh} and the references therein) and it currently arouses great interest in Convex Geometry and beyond. More recently, in \cite[Lemma~4.1, Proposition~4.2]{Sa}), the above inequality has been showed to be true in the case of unconditional convex bodies. The main idea of the proof is that it is enough to show the inequality for $A^{1-\lambda}B^{\lambda}$, where $A, B\subset\R^n_{>0}$, since
\begin{equation*}
\bigl(K\cap\R^n_{>0}\bigr)^{1-\lambda}\bigl(L\cap\R^n_{>0}\bigr)^{\lambda}\subset \bigl((1-\lambda)K+_0\lambda L\bigr)\cap\R^n_{>0}
\end{equation*}
together with the fact that an unconditional convex body is completely determined by its restriction to the positive orthant $\R^n_{>0}$.

In the same way, \eqref{e: Firey_BM-p} may be also obtained for the case $0<p<1$ when working with unconditional convex bodies (see \cite[Theorem~1.1]{Mar}). Once again it is enough to show it for $\bigl((1-\lambda)A^{p}+\lambda B^p\bigr)^{1/p}$, where $A, B\subset\R^n_{>0}$, because of the inclusion
\begin{equation*}
\Bigl((1-\lambda)\bigl(K\cap\R^n_{>0}\bigr)^{p}+\lambda \bigl(L\cap\R^n_{>0}\bigr)^p\Bigr)^{1/p}\subset \bigl((1-\lambda)K+_p\lambda L\bigr)\cap\R^n_{>0}.
\end{equation*}

\medskip

Here we provide with the proof of the above-mentioned inequalities for $A^{1-\lambda}B^{\lambda}$ and $\bigl((1-\lambda)A^{p}+\lambda B^p\bigr)^{1/p}$, respectively, for given sets $A, B\subset\R^n_{>0}$, as the aftermath of some results from Section \ref{s:main}.

First we show how, as a consequence of Corollary \ref{c:PL}, one can derive the following result (it can be found in \cite[Proposition~4.2]{Sa}) which, as we have commented above (and it is shown in \cite[Lemma~4.1, Proposition~4.2]{Sa}), allows us to deduce the log-Brunn-Minkowski inequality for unconditional convex bodies.
The proof we present here is essentially the same to the one therein, but we collect it because we think that by noticing that $A^{1-\lambda}B^\lambda$ is $(1-\lambda)A\star^\prime\lambda B$, for a convenient distance $\di$ on $\R_{>0}$, one can easily perceive how this result may be proven.
\begin{corollary}\label{c:logBM}
Let $\lambda\in(0,1)$ and let $A, B\subset\R_{>0}^n$ be measurable sets such that $A^{1-\lambda}B^\lambda$ is also measurable. Then
\begin{equation*}
\vol\bigl(A^{1-\lambda}B^\lambda\bigr)\geq\vol(A)^{1-\lambda}\vol(B)^{\lambda}.
\end{equation*}
\end{corollary}
\begin{proof}
Let $\di$ be the distance on $\R_{>0}$ defined (as in Definition \ref{d:+_d_assoc_to_phi}) via $\log(\cdot)$, i.e., $\di(x,y)=\enorm{\log(x)-\log(y)}$, and let $\star=\star_\di=+_\di$ be the operation in $\R_{>0}$ that we will consider.

Since $A^{1-\lambda}B^\lambda=(1-\lambda)A\star^\prime\lambda B$, where $\star^\prime$ (cf.~Proposition \ref{p:starProdSpac}) is the ``natural'' operation in $\R_{>0}^n$ as a product space of $(\R,\di)$, and by Corollary  \ref{c:PL}, it is enough to show that $(\R_{>0},\di)$ satisfies $\mathrm{PL}$ (and thus, it will also satisfy $\mathrm{BM}(0)$).

To this aim, let $f,g,h:\R_{>0}\longrightarrow\R_{\geq0}$ non-negative measurable functions such that
\begin{equation*}
h((1-\lambda)x+_d\lambda y)\geq f(x)^{1-\lambda}g(y)^{\lambda},
\end{equation*}
for all $x,y\in\R_{>0}$, and $\lambda\in(0,1)$. Then, denoting by $\widetilde{h}(\bar{x})=h\bigl(e^{\bar{x}}\bigr)$ (and analogously $\widetilde{f}$, $\widetilde{g}$) and $\bar{x}=\log(x)$, $\bar{y}=\log(y)$, we have
\begin{equation*}
\begin{split}
\widetilde{h}((1-\lambda)\bar{x}+\lambda \bar{y})&= h\bigl(e^{(1-\lambda)\log(x)+\lambda \log(y)}\bigr) =h((1-\lambda)x+_d\lambda y)\\
&\geq f(x)^{1-\lambda}g(y)^{\lambda}
=f\bigl(e^{\bar{x}}\bigr)^{1-\lambda}g\bigl(e^{\bar{y}}\bigr)^{\lambda}
=\widetilde{f}(\bar{x})^{1-\lambda}\widetilde{g}(\bar{y})^{\lambda}.
\end{split}
\end{equation*}
Thus, by Theorem \ref{t:PrekopaLeindler} (and using that $e^x$ is log-concave), we have
\begin{equation*}
\int_{\R}\widetilde{h}\,e^x\,\dlat\bar{x}\geq \left(\int_\R \widetilde{f}\,e^x\,\dlat\bar{x}\right)^{1-\lambda}
\left(\int_\R \widetilde{g}\,e^x\,\dlat\bar{x}\right)^{\lambda},
\end{equation*}
which means that
\begin{equation*}
\int_{\R_{>0}}h\,\dlat x\geq \left(\int_{\R_{>0}} f\,\dlat x\right)^{1-\lambda}
\left(\int_{\R_{>0}} g\,\dlat x\right)^{\lambda},
\end{equation*}
which concludes the proof.
\end{proof}

Now we deal with the Brunn-Minkowski inequality \eqref{e: Firey_BM-p} for the $L_p$-mean of unconditional convex bodies, for $p\in(0,1)$ (see \cite[Theorem~1.1]{Mar}). As we have previously commented, it is an immediate consequence of the following Brunn-Minkowski type inequality collected in Corollary \ref{c:p-suma}. To show it, we need the following auxiliary result.
\begin{lema}\label{l: p-BM}
Let $n\in\Z_{>0}$, $X=\R_{>0}$ and let $\di$ be the distance given by $\di(x,y)=\enorm{x^p-y^p}$, where $0<p<1$. Then the space $(X,\di,\vol_1)$ satisfies $\mathrm{BBL}(\alpha/(r\alpha+1),1)$ for $r=0,1,\dots,n-1$, where $\alpha=p/(n-np)$.
\end{lema}
\begin{proof}
Let $f,g,h:X\longrightarrow\R_{\geq0}$ non-negative measurable functions such that
\begin{equation*}
h((1-\lambda)x+_d\lambda y)\geq f(x)^{1-\lambda}g(y)^{\lambda},
\end{equation*}
for all $x,y\in X$, and $\lambda\in(0,1)$. Let $\widetilde{h}:\R\longrightarrow\R_{\geq0}$ be the function given by
\begin{equation*}
\widetilde{h}(\bar{x}) =\left\{\begin{array}{ll}
h\bigl(\bar{x}^{1/p}\bigr) & \text{ if } \, \bar{x}>0,\\[2mm]
0 & \text{ otherwise}
\end{array}\right.
\end{equation*}
(and let $\widetilde{f},\,\widetilde{g}:\R\longrightarrow\R_{\geq0}$ be the functions defined analogously with respect to $f$ and $g$).

The proof is now similar to that of Corollary \ref{c:logBM} but exchanging the geometric mean $M_0$ by the mean $M_{\alpha_r}$ where $\alpha_r=\alpha/(r\alpha+1)$ and applying the classical $\mathrm{BBL}$ inequality in $\R$, Theorem \ref{t:BBL}. Notice also that for this approach we need that $x^{1/p-1}$ is $\alpha_r$-concave (a fact that can be easily checked by testing that  $(1/p-1)\alpha_r\in(0,1)$) and that $\alpha_r\geq-1$ for $r=0,\dots,n-1$.
\end{proof}

\begin{corollary}\label{c:p-suma}
Let $\lambda\in(0,1)$, $p\in(0,1)$, and let $A, B\subset\R_{>0}^n$ be measurable sets such that $\bigl((1-\lambda)A^{p}+\lambda B^p\bigr)^{1/p}$ is also measurable. Then
\begin{equation*}\label{e:logBM}
\vol\left(\bigl((1-\lambda)A^{p}+\lambda B^p\bigr)^{1/p}\right)^{p/n}\geq(1-\lambda)\vol(A)^{p/n}+\lambda\vol(B)^{p/n}.
\end{equation*}
\end{corollary}
\begin{proof}
Let $\di$ be the distance on $\R_{>0}$ given by $\di(x,y)=\enorm{x^p-y^p}$. Then, by Lemma \ref{l: p-BM} and Corollary \ref{c:ndimBBL}, $(\R_{>0}^n, \di^n,\vol)$ satisfies $\mathrm{BBL}(\alpha,n)$ (notice that $\alpha=p/(n-np)\geq-1/n$).
Thus, and taking into account that $q=q(\alpha,n)=\alpha/(n\alpha+1)=p/n$, we may assert that $(\R_{>0}^n, \di^n,\vol)$ satisfies $\mathrm{BM}$($p/n$) (and thus also $\overline{\mathrm{BM}}$($p/n$), by Proposition \ref{p:BMzeromeas_sets}).
This concludes the proof.
\end{proof}


\end{document}